\documentclass[11pt]{article}
\usepackage{amssymb,amsmath}
\usepackage[mathscr]{eucal}
\usepackage[cm]{fullpage}
\usepackage[english]{babel}
\usepackage[latin1]{inputenc}
\usepackage{xcolor}
\usepackage{tikz}
\usepackage[scr=boondoxo]{mathalfa}
\usepackage{upgreek,float}

\title{On oriented alternating inverse monoids}

\author{
\textbf{V\'{\i}tor Hugo Fernandes}\\
Center for Mathematics and Applications (NOVA Math) 
and Department of Mathematics,  \\ 
Faculdade de Ci\^{e}ncias e Tecnologia, 
Universidade Nova de Lisboa, 
2829-516 Caparica, 
Portugal\\ 
E-mail: vhf@fct.unl.pt\\
ORCID iD: https://orcid.org/0000-0003-1057-4975\\ 
}

\newtheorem{theorem}{Theorem}[section]

\newtheorem{lemma}[theorem]{Lemma}
\newtheorem{proposition}[theorem]{Proposition}

\def\<{\leqslant}
\def\>{\geqslant}
\def\dom{\mathop{\mathrm{Dom}}\nolimits}
\def\im{\mathop{\mathrm{Im}}\nolimits}
 
\def\gd{\mathrm{d}}
\def\gi{\mathrm{i}}
\def\id{\mathrm{id}}

\def\Sym{\mathcal{S}}
\def\A{\mathcal{A}}
\def\POI{\mathcal{POI}}
\def\PMI{\mathcal{PMI}}
\def\POPI{\mathcal{POPI}}
\def\PORI{\mathcal{PORI}}
\def\I{\mathcal{I}}
\def\C{\mathcal{C}}
\def\D{\mathcal{D}}
\def\AI{\mathcal{AI}}
\def\AO{\mathcal{AO}} 
\def\AM{\mathcal{AM}} 
\def\AOP{\mathcal{AOP}} 
\def\AOR{\mathcal{AOR}} 
\def\con{\mathop{\mathrm{Con}}\nolimits} 
\newcommand{\transf}[1]{\left(\begin{smallmatrix} #1 \end{smallmatrix}\right)}
\renewcommand{\mod}[1]{\,(\mathrm{mod}{\,#1})}
\newcommand{\conpi}[1]{\uppi_{\mbox{$\!_{#1}$}}}
\newcommand{\contheta}[1]{\uptheta_{\mbox{$\!_{#1}$}}}
\newcommand{\rees}[1]{\sim_{\mbox{$\!_{#1}$}}} 
\newenvironment{proof}{\begin{trivlist}\item[\hskip%
\labelsep{\bf\em Proof.}]}%
{\qed\rm\end{trivlist}}
\newcommand{\qed}{{\unskip\nobreak
\hfil\penalty50\hskip .001pt \hbox{}
          \nobreak\hfil
         {\scriptsize$\Box$}
           \parfillskip=0pt\finalhyphendemerits=0\medbreak}}

\begin{document}

\maketitle 

\begin{abstract} 
In this paper, we consider the inverse submonoids $\AOR_n$ of oriented transformations and $\AOP_n$ of orientation-preserving transformations 
of the alternating inverse monoid $\AI_n$ on a chain with $n$ elements. We compute the cardinalities, describe the Green's structures and the congruences, and calculate the ranks of $\AOR_n$ and $\AOP_n$. 
\end{abstract}

\noindent{\small\it Keywords: \rm alternating partial permutations, oriented, orientation-preserving, congruences, rank.}  

\medskip 

\noindent{\small 2020 \it Mathematics subject classification: \rm 20M20, 20M18, 20M10.}  

\section*{Introduction and Preliminaries}\label{Int} 

Let $n$ be a positive integer and let $\Omega_n$ be a set with $n$ elements, e.g. $\Omega_n=\{1,2,\ldots,n\}$. 
We denote by $\Sym_n$  the \textit{symmetric group} on $\Omega_n$,
i.e. the group (under composition of mappings) of all permutations on $\Omega_n$, 
and by $\I_n$ the \textit{symmetric inverse monoid} on $\Omega_n$, i.e.
the inverse monoid (under composition of partial mappings) of all partial permutations (i.e. injective partial transformations) on $\Omega_n$.  

\smallskip 

Let $M$ be a monoid. The Green's equivalence relations $\mathcal{L}$, $\mathcal{R}$, $\mathcal{J}$ and $\mathcal{H}$ of $M$ are defined by
$a\mathcal{L}b$ if and only if $Ma=Mb$,  
$a\mathcal{R}b$ if and only if $aM=bM$,
$a\mathcal{J}b$ if and only if $MaM=MbM$ and 
$a\mathcal{H}b$ if and only if $a\mathcal{L}b$ and $a\mathcal{R}b$, for $a,b\in M$.  
Denote by $J_{a}^M$ the $\mathscr{J}$-class
of the element $a\in M$. As usual, a partial order relation
$\leqslant_\mathscr{J}$ is defined on the set $M/\mathscr{J}$ by
setting $J_{a}^M\leqslant_\mathscr{J}J_{b}^M$ if and only if 
$MaM\subseteq MbM$, for $a,b\in M$. We also 
write $J_{a}^M<_\mathscr{J}J_{b}^M$ if
and only if $J_a^M\leqslant_\mathscr{J}J_b^M$ and $(a, b)\not\in\mathscr{J}$, for $a, b\in M$. 
As usual, we also denote by $H_a^M$ the $\mathscr{H}$-class of an element $a\in M$. 

If $M$ is an inverse submonoid of $\I_n$,
it is well known that the Green's relations $\mathscr{L}$, $\mathscr{R}$ and $\mathscr{H}$ of $M$ can be described as following:
for $\alpha, \beta \in M$,
\begin{itemize}
\item $\alpha \mathscr{L} \beta$ if and only if $\im(\alpha) = \im(\beta)$,

\item $\alpha \mathscr{R} \beta$ if and only if $\dom(\alpha) = \dom(\beta)$, and

\item $\alpha \mathscr{H} \beta$ if and only if $\im(\alpha) = \im(\beta)$ and $\dom(\alpha) = \dom(\beta)$.
\end{itemize}
In particular for $\I_n$, we also have
\begin{itemize}
\item $\alpha \mathscr{J} \beta$ if and only if $|\dom(\alpha)| = |\dom(\beta)|$ (if and only if $|\im(\alpha)| = |\im(\beta)|$).
\end{itemize}
Let $J_k^{\I_n}=\{\alpha\in\I_n\mid |\im(\alpha)|=k\}$, for $0\leqslant k\leqslant n$. 
Then, we have $|J_k^{\I_n}|=\binom{n}{k}^2k!$, for $0\leqslant k\leqslant n$, and 
$
\I_n/\mathscr{J}=\{J_0^{\I_n}<_{\mathscr{J}}J_1^{\I_n}<_{\mathscr{J}}\cdots<_{\mathscr{J}}J_n^{\I_n}\}. 
$
Observe that $J_0^{\I_n}=\{\emptyset\}$ and $J_n^{\I_n}=\Sym_n$. 

\smallskip 

Let $G$ be a subgroup of $\Sym_n$ and let $\I_n(G)=\{\alpha\in\I_n\mid \mbox{$\alpha=\sigma|_{\dom(\alpha)}$, for some $\sigma\in G$}\}$.  
It is clear that $\I_n(G)$ is an inverse submonoid of $\I_n$ containing the semilattice $\mathcal{E}_n$ of all idempotents of $\I_n$ and with $G$ as group of units. 
For $G=\Sym_n$, $G=\A_n$ and $G=\{\id_n\}$,
where $\A_n$ denotes the \textit{alternating group} on $\Omega_n$ and $\id_n$ is the identity transformation of $\Omega_n$, 
we obtain important and well-known inverse submonoids of $\I_n$. In fact, on one hand, it is clear that 
$\I_n(\Sym_n)=\I_n$ and $\I_n(\{\id_n\})=\mathcal{E}_n$. 
On the other hand, we have 
$\I_n(\A_n)=\AI_n$, the \textit{alternating semigroup}; see \cite{Fernandes:2024sub}. 
Furthermore, the monoids $\I_n(\C_n)$ and $\I_n(\mathcal{D}_{2n})$, where $\C_n$ is a cyclic subgroup of $\Sym_n$ of order $n$ 
and $\mathcal{D}_{2n}$ is a dihedral subgroup of $\Sym_n$ of order $2n$,  
were studied in \cite{Fernandes:2024} and \cite{Fernandes&Paulista:2023}, respectively.  

\smallskip 

Next, we recall the structure of the alternating semigroup $\AI_n$ (denoted by $\A^c_n$ in \cite{Lipscomb:1996}).  
For $\alpha\in J_{n-1}^{\I_n}$, let us denote by $\overline{\alpha}$ the \textit{completion} of $\alpha$, 
i.e. $\overline{\alpha}$ is the unique permutation of 
$\Sym_n$ such that $\overline{\alpha}|_{\dom(\alpha)}=\alpha$. Let 
$J_{n-1}^{\A_n}=\{\alpha\in J_{n-1}^{\I_n}\mid \overline{\alpha}\in\A_n\}$. 
Notice that $|J_{n-1}^{\A_n}|=\frac{1}{2}n!n=\frac{1}{2}|J_{n-1}^{\I_n}|$. Then, we have 
$$
\AI_n=\A_n\cup J_{n-1}^{\A_n}\cup J_{n-2}^{\I_n}\cup\cdots\cup J_1^{\I_n}\cup J_0^{\I_n}
$$
(see \cite[Theorems 25.1 and 25.2]{Lipscomb:1996}) and so $|\AI_n|=\frac{1}{2}n!+\frac{1}{2}n!n+\sum_{k=0}^{n-2}\binom{n}{k}^2k!$. 

\medskip 

From now on, we will consider $\Omega_n$ as a chain, e.g. $\Omega_n=\{1<2<\cdots<n\}$. 
Let $\alpha\in\I_n$. 
We say that $\alpha$ is \textit{order-preserving}
[respectively, \textit{order-reversing}] if $x\leqslant y$ implies $x\alpha\leqslant y\alpha$
[respectively, $x\alpha\geqslant y\alpha$], for all $x,y \in \dom(\alpha)$.
A partial permutation is said to be \textit{monotone} if it is order-preserving or order-reversing.  
Suppose that $\dom(\alpha)=\{a_1<\cdots<a_t\}$, with $t\geqslant0$. 
We say that $\alpha$ is \textit{orientation-preserving}
[respectively, \textit{orientation-reversing}, \textit{oriented}]
if there exists no more than one index $i\in\{1,\ldots,t\}$ such that
$a_i\alpha >a_{i+1}\alpha$ [respectively, $a_i\alpha<a_{i+1}\alpha$],
where $a_{t+1}$ denotes $a_1$. 
Let us denote by $\POI_n$, $\PMI_n$, $\POPI_n$ and $\PORI_n$ 
the inverse submonoids of $\I_n$ 
of all order-preserving partial permutations, 
of all monotone partial permutations, 
of all orientation-preserving partial  permutations and 
of all oriented partial permutations, respectively.
Semigroups of monotone, order-preserving, orientation-preserving and oriented transformations 
have been studied massively in recent decades. 
See, for example, \cite{Dimitrova&Fernandes&Koppitz&Quinteiro:2023,Fernandes:2000,Fernandes:2001} and their references.  

\smallskip 

In \cite{Fernandes:2024sub}, the author considered the inverse submonoids $\AM_n=\AI_n\cap\PMI_n$ of monotone transformations 
and $\AO_n=\AI_n\cap\POI_n$ of order-preserving transformations of the alternating inverse monoid $\AI_n$. 
In that work, Fernandes computed the cardinalities, describe the Green's structures and the congruences, 
and calculate the ranks of $\AM_n$ and $\AO_n$.  
In this paper, we aim to carry out a similar study with regard to the inverse submonoids $\AOR_n=\AI_n\cap\PORI_n$ of oriented transformations 
and $\AOP_n=\AI_n\cap\POPI_n$ of orientation-preserving transformations of the alternating inverse monoid $\AI_n$. 
In Sections \ref{AOP} and \ref{AOR}, we study the structure of the monoids $\AOP_n$ and $\AOR_n$, respectively. Complete descriptions of the congruences of $\AOP_n$ and $\AOR_n$ are given in Section \ref{congruences}. Finally, in Sections \ref{rAOP} and \ref{rAOR}, we compute the ranks of $\AOP_n$ and $\AOR_n$, respectively. 

\smallskip 

We end this section by recalling some notions and by introducing some notation that we will use in the following sections.  

By \textit{rank} of a (finite) monoid $M$, we mean the minimum size of a generating set of $M$, i.e. 
the minimum of the set $\{|X|\mid \mbox{$X\subseteq M$ and $X$ generates $M$}\}$. 
As usual, we also use the term \textit{rank} relative to a transformation $\alpha$ of $\Omega_n$ to mean the number $|\im(\alpha)|$. 

An \textit{ideal} of $M$ is a subset $I$ of $M$ such that $MIM\subseteq I$.  
The \textit{Rees congruence} of $M$ associated to an ideal $I$ of $M$ 
is the congruence $\rees{I}$ is defined by $a\rees{I}b$ if and only if $a=b$ or
$a,b\in I$, for $a, b\in M$. 
We denote by $\omega$ the universal congruences of $M$, respectively. 
The congruence \textit{lattice} (for the inclusion order relation) of $M$ is denoted by $\con(M)$. 

For $\alpha\in J_{n-1}^{\I_n}$, define $\gd(\alpha)$ and $\gi(\alpha)$ as the \textit{gaps} of $\dom(\alpha)$ and $\im(\alpha)$, respectively, 
i.e. $\{\gd(\alpha)\}=\Omega_n\setminus\dom(\alpha)$ and $\{\gi(\alpha)\}=\Omega_n\setminus\im(\alpha)$.  

For subsets $A$ and $B$ of $\Omega_n$ such that $|A|=|B|$, 
let $\transf{A\\B}$ denote any transformation $\alpha\in\I_n$ such that $A=\dom(\alpha)$ and $B=\im(\alpha)$. 
Observe that, for any subsets $A$ and $B$ of $\Omega_n$ such that $|A|=|B|$, 
there exists a unique order-preserving partial permutation $\alpha$ of $\Omega_n$ such that $\alpha=\transf{A\\B}$. 

For a nonempty subset $X$ of $\Omega_n$, let us denote by $\id_X$ the partial identity $\id_n|_X$. 

\smallskip 

For general background on Semigroup Theory and standard notations, we refer to Howie's book \cite{Howie:1995}.

\smallskip 

In this paper, from now on, we will always consider $n\geqslant 3$. 

\section{The structure of $\AOP_n$} \label{AOP} 

Let $\C_n$ be the cyclic subgroup of $\Sym_n$ of order $n$ generated by the $n$-cycle
$$
g=\transf{1&2&\cdots&n-1&n\\2&3&\cdots&n&1} = \textstyle (1\:2\:\cdots\:n). 
$$

Recall that $\POPI_n/\mathscr{J}=\{J_0^{\POPI_n}<_{\mathscr{J}}J_1^{\POPI_n}<_{\mathscr{J}}\cdots<_{\mathscr{J}}J_n^{\POPI_n}\}$, 
where $J_k^{\POPI_n}=J_k^{\I_n}\cap\POPI_n$, for $0\leqslant k\leqslant n$. 
In this case, we have $|J_k^{\POPI_n}|=k\binom{n}{k}^2$, for $1\leqslant k\leqslant n$, with 
$J_0^{\POPI_n}=\{\emptyset\}$ and $J_n^{\POPI_n}=\C_n$. 
Moreover, the maximal subgroups of $J_k^{\POPI_n}$ are cyclic of order $k$,  for $1\leqslant k\leqslant n$. 
For more details, see \cite{Fernandes:2000}. 

Let $J_n=J_n^{\POPI_n}\cap\AOP_n=\C_n\cap\A_n$ (i.e. $J_n$ is the group of units of $\AOP_n$ and so also a $\mathscr{J}$-class and a $\mathscr{H}$-class). 
Noticing that $g=(1\:2)(1\:3)\cdots(1\:n)$, i.e. $g$ decomposes into $n-1$ transpositions, we can immediately conclude that, if $n$ is odd, then $\C_n\subseteq\A_n$. 
This same decomposition of $g$ also allows us to deduce that, if $n$ is even, then $g^k\in\A_n$ if and only if $k$ is even, for $k\>0$, and so, in this case, 
we get $\C_n\cap\A_n= \{g^{2k}\mid 0\< k\<\frac{n}{2}-1\}=\langle g^2\rangle$. 
Thus, we have
\begin{equation}\label{jn}
J_n=\left\{
\begin{array}{ll}
\C_n & \mbox{if $n$ is odd}\\
\langle g^2\rangle=\{g^{2k}\mid 0\< k\<\frac{n}{2}-1\} & \mbox{if $n$ is even}. 
\end{array}
\right.
\end{equation} 
On the other hand, obviously 
$$
\{\alpha\in\AOP_n\mid |\im(\alpha)|\leqslant n-2\}=\{\alpha\in\POPI_n\mid |\im(\alpha)|\leqslant n-2\}=J_{n-2}^{\POPI_n}\cup\cdots\cup J_1^{\POPI_n}\cup J_0^{\POPI_n}.
$$
Hence, it remains to characterize the elements of $\AOP_n$ with rank $n-1$, which we are going to do next. 

First, we recall the following two results proved in \cite{Fernandes:2024sub}. 

\begin{lemma}[{\cite[Proposition 1.1]{Fernandes:2024sub}}]\label{chAO}
Let $\alpha\in J_{n-1}^{\I_n}\cap\POI_n$. Then, $\alpha\in\AO_n$ if and only if $\gd(\alpha)$ and $\gi(\alpha)$ have the same parity. 
\end{lemma}

\begin{lemma}[{\cite[Lemma 2.1]{Fernandes:2024sub}}]\label{-ab}
Let $\alpha,\beta\in J_{n-1}^{\I_n}$ be such that $\alpha\beta\in J_{n-1}^{\I_n}$. 
Then, $\overline{\alpha\beta}=\overline\alpha\overline\beta$. 
\end{lemma}

Next, for $\alpha\in J_{n-1}^{\I_n}$, define $\alpha_L=\transf{\{1,\ldots,n-1\}\\\dom(\alpha)}\in\POI_n$, 
$\alpha_R=\transf{\im(\alpha)\\\{1,\ldots,n-1\}}\in\POI_n$ and $\widehat\alpha=\alpha_L\alpha\alpha_R$. 
Observe that, $\alpha_L\in\AO_n$ [respectively, $\alpha_R\in\AO_n$] if and only if $\gd(\alpha)$ [respectively, $\gi(\alpha)$] has the same parity as $n$. 
Define also 
$$
g_n=\transf{1&2&\cdots&n-2&n-1\\2&3&\cdots&n-1&1}.
$$
Notice that, $g_n\in J_{n-1}^{\POPI_n}$ and $\overline{g}_n=(1\:2)\cdots(1\:{n\!-\!1})$. 
Thus, by a reasoning analogous to the one above, we may conclude that 
\begin{equation}\label{gn}
\langle \overline{g}_n \rangle\cap\A_n =\left\{
\begin{array}{ll}
\langle \overline{g}_n \rangle& \mbox{if $n$ is even}\\
\langle \overline{g}_n^2 \rangle=\{\overline{g}_n^{2k}\mid 0\< k\<\frac{n-1}{2}-1\} & \mbox{if $n$ is odd}. 
\end{array}
\right.
\end{equation} 
Observe also that, if $\alpha\in J_{n-1}^{\POPI_n}$, then $\widehat\alpha\in\langle g_n\rangle$ and so, by (\ref{gn}) and Lemma \ref{-ab}, 
$\widehat\alpha\in\AOP_n$, if $n$ is even, and $\widehat\alpha\in\AOP_n$ if and only if $1\widehat\alpha$ is odd, if $n$ is odd. 

Now, we can show a description of the elements of $\AOP_n$ with rank $n-1$. 

\begin{proposition}\label{chAOP}
Let $\alpha\in J_{n-1}^{\POPI_n}$. Then:
\begin{enumerate}
\item For $n$ even, $\alpha\in\AOP_n$ if and only if $\gd(\alpha)$ and $\gi(\alpha)$ have the same parity;
\item For $n$ odd,  $\alpha\in\AOP_n$ if and only if 
either $\gd(\alpha)$ and $\gi(\alpha)$ have the same parity and $1\widehat\alpha$ is odd 
or $\gd(\alpha)$ and $\gi(\alpha)$ have distinct parities and $1\widehat\alpha$ is even.  
\end{enumerate} 
\end{proposition}
\begin{proof}
We begin by observing that $\alpha=\alpha_L^{-1}\widehat\alpha\alpha_R^{-1}$ and so, by Lemma \ref{-ab}, we have 
$\overline{\alpha}=\overline{\alpha}_L^{-1}\overline{\widehat\alpha}\overline{\alpha}_R^{-1}$. 
Observe also that $\gd(\alpha)$ and $\gi(\alpha)$ have the same parity 
if and only if $\overline{\alpha}_L,\overline{\alpha}_R\in\A_n$ or $\overline{\alpha}_L,\overline{\alpha}_R\notin\A_n$. 

First, suppose that $n$ is even. Then, $\overline{\widehat\alpha}\in\A_n$ and so 
$$
\begin{array}{rcl}
\alpha\in\AOP_n & \Leftrightarrow & \overline{\alpha}\in\A_n \\
& \Leftrightarrow & \mbox{$\overline{\alpha}_L^{-1}\in\A_n$ if and only if $\overline{\alpha}_R^{-1}\in\A_n$}\\
& \Leftrightarrow & \mbox{$\overline{\alpha}_L\in\A_n$ if and only if $\overline{\alpha}_R\in\A_n$}\\
& \Leftrightarrow & \mbox{$\gd(\alpha)$ and $\gi(\alpha)$ have the same parity.}\\
\end{array}
$$

Now, suppose that $n$ is odd. 
If $\gd(\alpha)$ and $\gi(\alpha)$ have the same parity, 
then 
$$
\alpha\in\AOP_n \quad \Leftrightarrow \quad \overline{\alpha}\in\A_n \quad \Leftrightarrow \quad \overline{\widehat\alpha}\in\A_n 
\quad \Leftrightarrow \quad \mbox{$1\widehat\alpha$ is odd}. 
$$ 
On the other hand, 
if $\gd(\alpha)$ and $\gi(\alpha)$ have the distinct parities (i.e. $\overline{\alpha}_L\in\A_n$ if and only if $\overline{\alpha}_R\notin\A_n$), 
then 
$$
\alpha\in\AOP_n \quad \Leftrightarrow \quad \overline{\alpha}\in\A_n \quad \Leftrightarrow \quad \overline{\widehat\alpha}\notin\A_n 
\quad \Leftrightarrow \quad \mbox{$1\widehat\alpha$ is even},  
$$ 
as required.
\end{proof} 

Observe that, in $\Omega_n$, we have 
$\frac{n}{2}$ even numbers and $\frac{n}{2}$ odd numbers, if $n$ is even, and 
$\frac{n-1}{2}$ even numbers and $\frac{n+1}{2}$ odd numbers, if $n$ is odd. 
Then, for $n$ even, we get $\frac{1}{2}n^2$ pairs with the same parity 
and, for $n$ odd, we obtain $\frac{n^2+1}{2}$ pairs with the same parity and $\frac{n^2-1}{2}$ pairs with distinct parities. 

Let $J_{n-1}=J_{n-1}^{\POPI_n}\cap\AOP_n$. 

Let us suppose that $n$ is even. Let $\alpha\in J_{n-1}$. Then, by Proposition \ref{chAOP}, $(-1)^{\gd(\alpha)}=(-1)^{\gi(\alpha)}$. 
On the other hand, if $\beta$ is one of the $n-1$ elements of $\POPI_n$ such that $\dom(\beta)=\dom(\alpha)$ and $\im(\beta)=\im(\alpha)$ 
(i.e. $\beta\mathscr{H}\alpha$ in $\POPI_n$), then $\gd(\beta)=\gd(\alpha)$ and $\gi(\beta)=\gi(\alpha)$ and so $\beta\in J_{n-1}$. 
Then, as we have $\frac{1}{2}n^2$ pairs with the same parity, it follows that 
$
|J_{n-1}|= \frac{n^2(n-1)}{2}. 
$

Next, suppose that $n$ is odd. Let $\alpha\in J_{n-1}^{\POPI_n}$. Then, as clearly the mapping 
$$
\begin{array}{ccc}
\{\beta\in\POPI_n\mid \beta\mathscr{H}\alpha\}&\longrightarrow&  \langle g_n\rangle = \{\gamma\in\POPI_n\mid \gamma\mathscr{H} g_n\}  \\
\beta & \longmapsto & \widehat\beta
\end{array}
$$
is a bijection and 
$$
|\{\gamma\in\langle g_n\rangle \mid \mbox{$1\gamma$ is odd}\}| = |\{\gamma\in\langle g_n\rangle \mid \mbox{$1\gamma$ is even}\}| = \textstyle \frac{n-1}{2}, 
$$
we have 
$$
|\{\beta\in\POPI_n\mid \mbox{$\beta\mathscr{H}\alpha$ and $1\widehat\beta$ is odd}\}| = 
|\{\beta\in\POPI_n\mid \mbox{$\beta\mathscr{H}\alpha$ and $1\widehat\beta$ is even}\}| = \textstyle \frac{n-1}{2}.  
$$
Since $J_{n-1}^{\POPI_n}$ contains $n^2$ $\mathscr{H}$-classes, by Proposition \ref{chAOP}, it follows again that 
$
|J_{n-1}|= \frac{n^2(n-1)}{2}. 
$

Therefore, as $|\POPI_n|=1+\frac{n}{2}\binom{2n}{n}$, given the previous calculations and (\ref{jn}), we obtain: 

\begin{proposition}\label{carAOP}
$|\AOP_n| = \left\{
\begin{array}{ll}
\frac{n}{2}\binom{2n}{n} -  \frac{n^2(n-1)}{2} +1 & \mbox{if $n$ is odd}\\ \\ 
\frac{n}{2}\binom{2n}{n} -  \frac{n^2(n-1)+n}{2} +1 & \mbox{if $n$ is even}. 
\end{array}
\right.
$
\end{proposition}

For $0\leqslant k\leqslant n-2$, let us denote $J_k^{\POPI_n}$ simply by $J_k$. It is easy to conclude that $J_0,J_1,\ldots,J_{n-2}$ 
are also $\mathscr{J}$-classes of $\AOP_n$. Moreover, we have $J_0<_\mathscr{J}J_1<_\mathscr{J}\cdots<_\mathscr{J}J_{n-2}<_\mathscr{J}J_n$  in $\AOP_n$. 
Regarding the elements of $\AOP_n$ with rank $n-1$, we first prove two lemmas. 

\begin{lemma}\label{jn-1o}
If $n$ is an odd number, then $J_{n-1}=J_n\langle g_n^2\rangle J_n$. 
\end{lemma}
\begin{proof}
First, notice that, in this case, $g_n^2\in\AOP_n$ and $J_n=\C_n$.  
Hence, $J_n\langle g_n^2\rangle J_n\subseteq\AOP_n$ and so, as the elements of $J_n$ are permutations of $\Omega_n$, 
we have $J_n\langle g_n^2\rangle J_n\subseteq J_{n-1}$. 
On the other hand, 
$$
\dom(g^rg_n^{2k}g^s)=\Omega_n\setminus\{n-r\}\quad\text{and}\quad\im(g^rg_n^{2k}g^s)=\Omega_n\setminus\{s\},
$$
for $0\< r,s\< n-1$ and $0\< k\< \frac{n-3}{2}$, 
from which it is easy to deduce that $|J_n\langle g_n^2\rangle J_n|=|\C_n\langle g_n^2\rangle\C_n|=n\cdot\frac{n-1}{2}\cdot n=|J_{n-1}|$, 
whence $J_n\langle g_n^2\rangle J_n=J_{n-1}$, as required. 
\end{proof}

Now, define 
$$
g_1=\transf{2&3&\cdots&n-1&n\\3&4&\cdots&n&2}, \quad 
J_{n-1}^\mathscr{o}=\{\alpha\in J_{n-1}\mid \mbox{$\gd(\alpha)$ is odd}\}
\quad\text{and}\quad 
J_{n-1}^\mathscr{e}=\{\alpha\in J_{n-1}\mid \mbox{$\gd(\alpha)$ is even}\}. 
$$
Then, we have: 

\begin{lemma}\label{jn-1e}
If $n$ is an even number, then $J_{n-1}^\mathscr{o}=J_n \langle g_1\rangle J_n$ 
and $J_{n-1}^\mathscr{e}=J_n \langle g_n\rangle J_n$.
\end{lemma}
\begin{proof}
Observe that, in this case, we get $J_n=\langle g^2\rangle$ and, clearly, $g_1,g_n\in\AOP_n$. 
Moreover, $g_1\in J _{n-1}^\mathscr{o}$ and $g_n\in J _{n-1}^\mathscr{e}$. 

Let $0\< r,s\< \frac{n-2}{2}$ and $0\< k\< n-2$. Then, we have 
$$
\dom(g^{2r}g_1^{k}g^{2s})=\Omega_n\setminus\{n-2r+1\}\quad\text{and}\quad\im(g^{2r}g_1^{k}g^{2s})=\Omega_n\setminus\{2s+1\}
$$
and
$$
\dom(g^{2r}g_n^{k}g^{2s})=\Omega_n\setminus\{n-2r\}\quad\text{and}\quad\im(g^{2r}g_n^{k}g^{2s})=\Omega_n\setminus\{2s\}. 
$$
Hence, it follows immediately that $J_n \langle g_1\rangle J_n\subseteq J_{n-1}^\mathscr{o}$ 
and $J_n \langle g_n\rangle J_n\subseteq J_{n-1}^\mathscr{e}$. 
On the other hand, the above equalities also allow us to deduce that 
$|J_n\langle g_1\rangle J_n|=\frac{n}{2}\cdot(n-1)\cdot\frac{n}{2}=|J_{n-1}^\mathscr{o}|$ 
and $|J_n\langle g_n\rangle J_n|=\frac{n}{2}\cdot(n-1)\cdot\frac{n}{2}=|J_{n-1}^\mathscr{e}|$, 
whence $J_n \langle g_1\rangle J_n = J_{n-1}^\mathscr{o}$ 
and $J_n \langle g_n\rangle J_n = J_{n-1}^\mathscr{e}$, as required. 
\end{proof}

Now, we can establish the following result. 

\begin{proposition}\label{relJ}
Let $\alpha,\beta\in\AOP_n$ be two elements of rank $n-1$. Then:
\begin{enumerate}
\item For $n$ odd,  $\alpha\mathscr{J}\beta$; 
\item For $n$ even,  $\alpha\mathscr{J}\beta$ if and only if $\gd(\alpha)$ and $\gd(\beta)$ have the same parity. 
\end{enumerate}
\end{proposition}

\begin{proof} First, suppose that $n$ is odd. Then, by Lemma \ref{jn-1o}, we have $J_{n-1}=\C_n\langle g_n^2\rangle \C_n$, 
whence $\alpha=g^r (g_n^2)^{k} g^s$, for some $0\< r,s\< n-1$ and $0\< k\< \frac{n-3}{2}$. 
Thus, $(g_n^2)^k=g^{n-r}\alpha g^{n-s}$ and so $g_n^2 = g^{n-s} \alpha g^{n-r} (g_n^2)^{\frac{n-1}{2}-k+1}$. 
Therefore, $\alpha\mathscr{J}g_n^2$. Analogously, $\beta\mathscr{J}g_n^2$ and so $\alpha\mathscr{J}\beta$. 

\smallskip 

Next, we suppose that $n$ is even. 
If $\alpha\in J_{n-1}^\mathscr{o}$, then using Lemma \ref{jn-1e} and performing calculations similar to those in the previous paragraph,  
we can show that $\alpha\mathscr{J}g_1$. Analogously, if $\alpha\in J_{n-1}^\mathscr{e}$, then $\alpha\mathscr{J}g_n$. 
Thus, by showing that $(g_1,g_n)\not\in\mathscr{J}$, the proof is complete. 
So, suppose there exist $\gamma,\lambda\in\AOP_n$ such that $g_1=\gamma g_n \lambda$ and assume, 
without loss of generality, that $\gamma$ also has rank $n-1$. Then, $\dom(g_1)=\dom(\gamma)$ and $\im(\gamma)=\dom(g_n)$,  
whence $1=\gd(g_1)=\gd(\gamma)$ and $\gi(\gamma)=\gd(g_n)=n$, 
which is a contradiction since $\gd(\gamma)$ and $\gi(\gamma)$ should have the same parity. 
Therefore, $[g_1]_\mathscr{J}\not\leqslant_\mathscr{J}[g_n]_\mathscr{J}$ and so $(g_1,g_n)\not\in\mathscr{J}$, as required. 
 \end{proof} 
 
Let us suppose that $n$ is odd. Then, by the previous proposition, $J_{n-1}$ 
is a $\mathscr{J}$-class of $\AOP_n$. Moreover, 
it is not difficult to deduce that $J_{n-1}$ has $n$ $\mathscr{L}$-classes (and $\mathscr{R}$-classes) and its  
$\mathscr{H}$-classes have size $\frac{n-1}{2}$ (in particular, in the case where they contain an idempotent, they are cyclic groups of order $\frac{n-1}{2}$). 
Furthermore, we get $J_0<_\mathscr{J}J_1<_\mathscr{J}\cdots<_\mathscr{J}J_{n-2}<_\mathscr{J}J_{n-1}<_\mathscr{J}J_n$ in $\AOP_n$. 

\smallskip 

On the other hand, if $n$ is even, 
then Proposition \ref{relJ} enables us to conclude that $\AOP_n$ has two $\mathscr{J}$-classes of elements of rank $n-1$, namely,
$J_{n-1}^\mathscr{o}$ and $J_{n-1}^\mathscr{e}$. 
Moreover, both $J_{n-1}^\mathscr{o}$  and $J_{n-1}^\mathscr{e}$ have $\frac{n}{2}$ $\mathscr{L}$-classes (and $\mathscr{R}$-classes) and 
their $\mathscr{H}$-classes have size $n-1$ (in particular, their $\mathscr{H}$-classes of idempotents are cyclic groups of order $n-1$). 
Furthermore, it is easy to check that the poset of $\mathscr{J}$-classes of $\AOP_n$ can be represented by the Hasse diagram of Figure \ref{fAOP}. 
\begin{figure}[ht] 
\centering
\begin{tikzpicture}[scale=0.5]
\draw (0,0) node{\fbox{$J_n$}} ; 
\draw (-2,-1.5) node{\fbox{$J_{n-1}^\mathscr{o}$}} ; 
\draw (2,-1.5) node{\fbox{$J_{n-1}^\mathscr{e}$}} ; 
\draw (0,-3) node{\fbox{$J_{n-2}$}} ; 
\draw (0,-5) node{\fbox{$J_1$}} ; 
\draw (0,-6.5) node{\fbox{$J_0$}} ; 
\draw[thin] (0,-.55) -- (-1,-0.89);
\draw[thin] (0,-.55) -- (1,-0.89);
\draw[thin] (0,-2.42) -- (1,-2.11);
\draw[thin] (0,-2.42) -- (-1,-2.11);
\draw[thin,dotted] (0,-3.65) -- (0,-4.35);
\draw[thin] (0,-5.55) -- (0,-5.95);
\end{tikzpicture}
\caption{The Hasse diagram of $\AOP_n/_\mathscr{J}$, for $n$ even.} \label{fAOP}
\end{figure}

\section{The structure of $\AOR_n$}\label{AOR}

Let us consider the following permutation $h$ of order $2$ of $\Sym_n$: 
$$
h=\transf{1&2&\cdots&n-1&n\\n&n-1&\cdots&2&1} = \textstyle (1\:n)(2\:n)\cdots(\lfloor\frac{n}{2}\rfloor\:\lfloor\frac{n}{2}\rfloor+1). 
$$ 
Then,  for $n\geqslant3$, $h$ and $g$ generate the well-known \textit{dihedral group} $\D_{2n}$ of order $2n$ 
(considered as a subgroup of $\Sym_n$). In fact, we have 
$$
\D_{2n}=\langle g,h\mid g^n=1,h^2=1, hg=g^{n-1}h\rangle=\{\id_n,g,g^2,\ldots,g^{n-1}, h,hg,hg^2,\ldots,hg^{n-1}\}. 
$$
Since it is clear that $h\in\AI_n$ if and only if $\lfloor\frac{n}{2}\rfloor$ is even if and only if $n\equiv 0\mod 4$ or $n\equiv 1\mod 4$, 
in view of (\ref{jn}), it is easy to conclude that 
\begin{equation}\label{qn}
\D_{2n}\cap\A_n=\left\{
\begin{array}{ll}
\{g^{2k}, hg^{2k}\mid 0\< k\<\frac{n}{2}-1 \}  = \langle g^2,h\rangle   & \mbox{if $n\equiv 0\mod 4$}\\
\D_{2n} & \mbox{if $n\equiv 1\mod 4$}\\
\{g^{2k}, hg^{2k+1}\mid 0\< k\<\frac{n}{2}-1 \} = \langle g^2,hg\rangle   & \mbox{if $n\equiv 2\mod 4$}\\
\C_{n} & \mbox{if $n\equiv 3\mod 4$}. 
\end{array}
\right.
\end{equation} 
Notice that, for $n$ even, both $\langle g^2,h\rangle$ and $\langle g^2,hg\rangle$ are dihedral groups of order $n\,(=2\frac{n}{2})$. 

\medskip 

Now, recall that $\PORI_n/\mathscr{J}=\{J_0^{\PORI_n}<_{\mathscr{J}}J_1^{\PORI_n}<_{\mathscr{J}}\cdots<_{\mathscr{J}}J_n^{\PORI_n}\}$, 
where $J_k^{\PORI_n}=J_k^{\I_n}\cap\PORI_n$, for $0\leqslant k\leqslant n$. 
We have $|J_0^{\PORI_n}|=1$, $|J_1^{\PORI_n}|=n^2$, $|J_2^{\PORI_n}|=2\binom{n}{2}^2$ and $|J_k^{\PORI_n}|=2k\binom{n}{k}^2$, for $3\leqslant k\leqslant n$. 
In particular,  $J_0^{\PORI_n}=\{\emptyset\}$ and $J_n^{\PORI_n}=\D_{2n}$. 
Moreover, the maximal groups of $J_k^{\PORI_n}$ are dihedral of order $2k$ for $3\leqslant k\leqslant n$, of order $2$ for $k=2$, and trivial for $k\leqslant 1$.  
For more details, see \cite{Fernandes&Gomes&Jesus:2004}. 

\medskip 

From now on, in this section, we consider $n\geqslant 4$. Observe that, $\AOR_3=\AOP_3=\AI_3$. 

\smallskip 

Let $Q_n$  be the group of units of $\AOR_n$. Then, $Q_n=J_n^{\PORI_n}\cap\AOR_n=\D_{2n}\cap\A_n$. 
On the other hand, we have 
$$
\{\alpha\in\AOR_n\mid |\im(\alpha)|\leqslant n-2\}=\{\alpha\in\PORI_n\mid |\im(\alpha)|\leqslant n-2\}=J_{n-2}^{\PORI_n}\cup\cdots\cup J_1^{\PORI_n}\cup J_0^{\PORI_n}.
$$
Therefore, as for $\AOP_n$, it remains to characterize the elements of $\AOR_n$ with rank $n-1$. 
More precisely, as $\POPI_n\subseteq\PORI_n$, it follows that $\AOP_n\subseteq\AOR_n$ and so, in view of Proposition \ref{chAOP}, 
we only need to characterize the elements of $\AOR_n\setminus\AOP_n$ with rank $n-1$. 

\medskip 

Next, let us consider 
$$
h_n=\transf{1&2&\cdots&n-2&n-1\\n-1&n-2&\cdots&2&1}\in J_{n-1}^{\PORI_n}. 
$$ 
Then, $\langle g_n,h_n\rangle$ is a dihedral group of order $2(n-1)$ and an $\mathscr{H}$-class of $\PORI_n$. 
As $\bar h_n\in\A_n$ if and only if $n-1\equiv 0\mod 4$ or $n-1\equiv 1\mod 4$  if and only if $n\equiv 1\mod 4$ or $n\equiv 2\mod 4$, 
in view of (\ref{gn}), we may deduce that 
\begin{equation}\label{gnhn}
\langle \bar g_n,\bar h_n\rangle\cap\A_n=\left\{
\begin{array}{ll}
\langle \bar g_n\rangle  & \mbox{if $n\equiv 0\mod 4$}\\
\langle \bar g_n^2,\bar h_n\rangle  & \mbox{if $n\equiv 1\mod 4$}\\
\langle \bar g_n,\bar h_n\rangle  & \mbox{if $n\equiv 2\mod 4$}\\
\langle \bar g_n^2,\bar h_n\bar g_n\rangle  & \mbox{if $n\equiv 3\mod 4$}. 
\end{array}
\right.
\end{equation} 
Observe that, if $n$ is odd, then $\langle \bar g_n^2,\bar h_n\rangle$ and $\langle \bar g_n^2,\bar h_n \bar g_n\rangle$ are dihedral groups of order $n-1\,(=2\frac{n-1}{2})$. 

Let $\alpha\in J_{n-1}^{\PORI_n}$. Then, we have $\widehat\alpha\in \langle g_n,h_n\rangle$. 
Since $\alpha_L,\alpha_R\in\POI_n$, if $\alpha\in J_{n-1}^{\PORI_n}\setminus J_{n-1}^{\POPI_n}$, then $\widehat\alpha\in J_{n-1}^{\PORI_n}\setminus J_{n-1}^{\POPI_n}$, 
and so $\widehat\alpha\in\langle g_n,h_n\rangle\setminus\langle g_n\rangle=\{h_n,h_ng_n,\ldots,h_ng_n^{n-2}\}$. 
Therefore, for  $\alpha\in J_{n-1}^{\PORI_n}\setminus J_{n-1}^{\POPI_n}$, 
it follows from (\ref{gnhn}) and Lemma \ref{-ab} that: 
\begin{equation}\label{qn-1}
\begin{array}{l}
\mbox{if $n\equiv 0\mod 4$, then $\widehat\alpha\not\in\AOR_n$;}\\
\mbox{if $n\equiv 1\mod 4$, then $\widehat\alpha\in\AOR_n$ if and only if $1\widehat\alpha$ is even;}   \\
\mbox{if $n\equiv 2\mod 4$, then $\widehat\alpha\in\AOR_n$;}\\
\mbox{if $n\equiv 3\mod 4$, then $\widehat\alpha\in\AOR_n$ if and only if $1\widehat\alpha$ is odd.}  
\end{array}
\end{equation} 

We can now prove the following description of the elements of $\AOR_n\setminus\AOP_n$ with rank $n-1$: 

\begin{proposition}\label{chAOR}
Let $\alpha\in J_{n-1}^{\PORI_n}\setminus J_{n-1}^{\POPI_n}$. Then:
\begin{enumerate}
\item For $n\equiv 0\mod 4$,  $\alpha\in\AOR_n$ if and only if $\gd(\alpha)$ and $\gi(\alpha)$ have distinct parities;
\item For $n\equiv 1\mod 4$,  $\alpha\in\AOR_n$ if and only if 
either $\gd(\alpha)$ and $\gi(\alpha)$ have the same parity and $1\widehat\alpha$ is even  
or $\gd(\alpha)$ and $\gi(\alpha)$ have distinct parities and $1\widehat\alpha$ is odd.  
\item For $n\equiv 2\mod 4$,  $\alpha\in\AOR_n$ if and only if $\gd(\alpha)$ and $\gi(\alpha)$ have the same parity;
\item For $n\equiv 3\mod 4$,  $\alpha\in\AOR_n$ if and only if  
either $\gd(\alpha)$ and $\gi(\alpha)$ have the same parity and $1\widehat\alpha$ is odd 
or $\gd(\alpha)$ and $\gi(\alpha)$ have distinct parities and $1\widehat\alpha$ is even.  
\end{enumerate}
\end{proposition}
\begin{proof} 

We begin by recalling that $\alpha=\alpha_L^{-1}\widehat\alpha\alpha_R^{-1}$ and so, by Lemma \ref{-ab}, 
$\overline{\alpha}=\overline{\alpha}_L^{-1}\overline{\widehat\alpha}\overline{\alpha}_R^{-1}$. 
Recall also that $\gd(\alpha)$ and $\gi(\alpha)$ have the same parity 
if and only if $\overline{\alpha}_L,\overline{\alpha}_R\in\A_n$ or $\overline{\alpha}_L,\overline{\alpha}_R\notin\A_n$. 

First, we take $n\equiv 0\mod 4$. Then, by (\ref{qn-1}), we have $\overline{\widehat\alpha}\not\in\A_n$ and so 
$$
\begin{array}{rcl}
\alpha\in\AOR_n & \Leftrightarrow & \overline{\alpha}\in\A_n \\
& \Leftrightarrow & \mbox{$\overline{\alpha}_L^{-1}\in\A_n$ if and only if $\overline{\alpha}_R^{-1}\notin\A_n$}\\
& \Leftrightarrow & \mbox{$\overline{\alpha}_L\in\A_n$ if and only if $\overline{\alpha}_R\notin\A_n$}\\
& \Leftrightarrow & \mbox{$\gd(\alpha)$ and $\gi(\alpha)$ have distinct parities.}\\
\end{array}
$$

Similarly, for $n\equiv 2\mod 4$, by (\ref{qn-1}), we have $\overline{\widehat\alpha}\in\A_n$ and so 
$$
\begin{array}{rcl}
\alpha\in\AOR_n & \Leftrightarrow & \overline{\alpha}\in\A_n \\
& \Leftrightarrow & \mbox{$\overline{\alpha}_L^{-1}\in\A_n$ if and only if $\overline{\alpha}_R^{-1}\in\A_n$}\\
& \Leftrightarrow & \mbox{$\overline{\alpha}_L\in\A_n$ if and only if $\overline{\alpha}_R\in\A_n$}\\
& \Leftrightarrow & \mbox{$\gd(\alpha)$ and $\gi(\alpha)$ have the same parity.}\\
\end{array}
$$

Next, let us suppose that $n\equiv 1\mod 4$. 
If $\gd(\alpha)$ and $\gi(\alpha)$ have the same parity, by (\ref{qn-1}), we have  
$$
\alpha\in\AOR_n \quad \Leftrightarrow \quad \overline{\alpha}\in\A_n \quad \Leftrightarrow \quad \overline{\widehat\alpha}\in\A_n 
\quad \Leftrightarrow \quad \mbox{$1\widehat\alpha$ is even}. 
$$ 
On the other hand, by (\ref{qn-1}), 
if $\gd(\alpha)$ and $\gi(\alpha)$ have the distinct parities, then 
then 
$$
\alpha\in\AOP_n \quad \Leftrightarrow \quad \overline{\alpha}\in\A_n \quad \Leftrightarrow \quad \overline{\widehat\alpha}\notin\A_n 
\quad \Leftrightarrow \quad \mbox{$1\widehat\alpha$ is odd}.  
$$ 

Similarly, for $n\equiv 3\mod 4$, by (\ref{qn-1}), we have:  
if $\gd(\alpha)$ and $\gi(\alpha)$ have the same parity, then 
$$
\alpha\in\AOR_n \quad \Leftrightarrow \quad \overline{\alpha}\in\A_n \quad \Leftrightarrow \quad \overline{\widehat\alpha}\in\A_n 
\quad \Leftrightarrow \quad \mbox{$1\widehat\alpha$ is odd}; 
$$ 
if $\gd(\alpha)$ and $\gi(\alpha)$ have the distinct parities, then 
then 
$$
\alpha\in\AOP_n \quad \Leftrightarrow \quad \overline{\alpha}\in\A_n \quad \Leftrightarrow \quad \overline{\widehat\alpha}\notin\A_n 
\quad \Leftrightarrow \quad \mbox{$1\widehat\alpha$ is even},  
$$ 
as required. 
\end{proof}

Let $Q_{n-1}=J_{n-1}^{\PORI_n}\cap\AOR_n$. 

Suppose that $n$ is even. Hence, in $\Omega_n$, we have $\frac{1}{2}n^2$ pairs with the same parity and $\frac{1}{2}n^2$ pairs with distinct parities. 
On the other hand, for each pair $(d,i)$ of elements of $\Omega_n$, we have $n-1$ transformations $\alpha\in\POPI_n$ such that $\gd(\alpha)=d$ and $\gi(\alpha)=i$, 
and $n-1$ transformations $\alpha\in\PORI_n\setminus\POPI_n$ such that $\gd(\alpha)=d$ and $\gi(\alpha)=i$. Hence, by Propositions \ref{chAOP} and \ref{chAOR}, 
we obtain $\frac{1}{2}n^2(n-1)+\frac{1}{2}n^2(n-1)$ elements in $Q_{n-1}$, i.e. $|Q_{n-1}|=n^2(n-1)$. 

Next, suppose that $n$ is odd. 
Let $\alpha\in J_{n-1}^{\PORI_n}$. Then, clearly, the mapping 
$$
\begin{array}{ccc}
\{\beta\in\PORI_n\setminus\POPI_n\mid \beta\mathscr{H}\alpha\}&\longrightarrow&  \langle g_n,h_n\rangle\setminus \langle g_n\rangle = 
\{\gamma\in\PORI_n\setminus\POPI_n\mid \gamma\mathscr{H} g_n\}  \\
\beta & \longmapsto & \widehat\beta
\end{array}
$$
(considering the relation $\mathscr{H}$ in $\PORI_n$) is a bijection and 
$$
|\{\gamma\in\langle g_n,h_n\rangle\setminus\langle g_n\rangle \mid \mbox{$1\gamma$ is odd}\}| = 
|\{\gamma\in\langle g_n,h_n\rangle\setminus \langle g_n\rangle \mid \mbox{$1\gamma$ is even}\}| = \textstyle \frac{n-1}{2}, 
$$
whence 
$$
|\{\beta\in\PORI_n\setminus\POPI_n\mid \mbox{$\beta\mathscr{H}\alpha$ and $1\widehat\beta$ is odd}\}| = 
|\{\beta\in\PORI_n\setminus\POPI_n\mid \mbox{$\beta\mathscr{H}\alpha$ and $1\widehat\beta$ is even}\}| = \textstyle \frac{n-1}{2}.  
$$
Since $J_{n-1}^{\PORI_n}$ contains $n^2$ $\mathscr{H}$-classes, by Proposition \ref{chAOR}, it follows that 
$
|Q_{n-1}\cap(\PORI_n\setminus\POPI_n)|= \frac{n^2(n-1)}{2}. 
$
On the other hand, as proved in Section \ref{AOP}, $|Q_{n-1}\cap\POPI_n|=|J_{n-1}|=\frac{n^2(n-1)}{2}$, whence 
$
|Q_{n-1}|= n^2(n-1). 
$

Therefore, as $|\PORI_n|=1+n\binom{2n}{n} -\frac{n^2(n^2-2n+3)}{2}$, by the previous calculations and (\ref{qn}), we obtain: 

\begin{proposition}\label{carAOR}
For $n\geqslant4$, 
$|\AOR_n| = \left\{
\begin{array}{ll}
1+n\binom{2n}{n} -\frac{n^2(n^2+1)}{2} & \mbox{if $n\equiv 1\mod 4$}\\ \\ 
1+n\binom{2n}{n} -\frac{n^2(n^2+1)}{2} -n & \mbox{if $n\not\equiv 1\mod 4$}. 
\end{array}
\right.
$
\end{proposition}

\medskip 

For $0\leqslant k\leqslant n-2$, denote $J_k^{\PORI_n}$ simply by $Q_k$. Then, as for $\AOP_n$, 
it is easy to conclude that $Q_0,Q_1,\ldots,Q_{n-2}$ are also $\mathscr{J}$-classes of $\AOR_n$ and 
$Q_0<_\mathscr{J}Q_1<_\mathscr{J}\cdots<_\mathscr{J}Q_{n-2}<_\mathscr{J}Q_n$  in $\AOR_n$. 
With respect to the elements of $\AOR_n$ with rank $n-1$, we have the following result. 

\begin{proposition}\label{RrelJ}
Let $\alpha,\beta\in\AOR_n$ be two elements of rank $n-1$. Then:
\begin{enumerate}
\item For $n\not\equiv 2\mod 4$,  $\alpha\mathscr{J}\beta$; 
\item For $n\equiv 2\mod 4$,  $\alpha\mathscr{J}\beta$ if and only if $\gd(\alpha)$ and $\gd(\beta)$ have the same parity. 
\end{enumerate}
\end{proposition}
\begin{proof}
First, we suppose that $n$ is odd. 

If $\alpha\in\AOP_n$, then $\alpha\mathscr{J}g_n^2$ in $\AOP_n$, 
by Proposition \ref{relJ}, and so $\alpha\mathscr{J}g_n^2$ (in $\AOR_n$). 

So, suppose that $\alpha\in\AOR_n\setminus\AOP_n$. 
Define $h_0=h_n$, if $n\equiv 1\mod 4$, and $h_0=h_ng_n$, if $n\equiv 3\mod 4$. 
Then, $h_0\in Q_{n-1}$,  $\alpha g^{n-\gi(\alpha)}h_0\in J_{n-1}$ and $\alpha=\alpha g^{n-\gi(\alpha)}h_0^2g^{\gi(\alpha)}$. 
Observe that $g\in\AOR_n$. 
Therefore, by Proposition \ref{relJ}, we have $\alpha g^{n-\gi(\alpha)}h_0\mathscr{J}g_n^2$ in $\AOP_n$, whence 
$\alpha g^{n-\gi(\alpha)}h_0=ug_n^2v$ and $g_n^2=u'\alpha g^{n-\gi(\alpha)}h_0v'$, for some $u,v,u',v'\in\AOP_n$. 
Thus, $\alpha=ug_n^2 v h_0g^{\gi(\alpha)}$ and so $\alpha\mathscr{J}g_n^2$ (in $\AOR_n$). 

Analogously, $\beta\mathscr{J}g_n^2$ and so $\alpha\mathscr{J}\beta$. 

\smallskip

Next, we suppose that $n\equiv 0\mod 4$. 
We begin by showing that, in this case, $g_1\mathscr{J}g_n$ (in $\AOR_n$). 
Let 
$$
g_{n,1}=\transf{1&2&\cdots&n-2&n-1\\n&n-1&\cdots&3&2}\quad\text{and}\quad  g_{1,n}=\transf{2&3&4&\cdots&n\\2&1&n-1&\cdots&3}. 
$$
Then, $g_{n,1},g_{1,n}\in\AOR_n$, by Proposition \ref{chAOR}, and $g_n=g_{n,1}g_1g_{1,n}$ and $g_1=g_{n,1}^{-1}g_ng_{1,n}^{-1}$, 
whence $g_1\mathscr{J}g_n$. 

If $\alpha\in\AOP_n$, then $\alpha\mathscr{J}g_n$ or $\alpha\mathscr{J}g_1$ in $\AOP_n$, 
by Proposition \ref{relJ}, and so $\alpha\mathscr{J}g_n$ (in $\AOR_n$), since $g_1\mathscr{J}g_n$. 

So, suppose that $\alpha\in\AOR_n\setminus\AOP_n$. As $h\in\AOR_n$, then $\alpha h\in J_{n-1}$, 
whence $\alpha h\mathscr{J}g_n$ or $\alpha h\mathscr{J}g_1$ in $\AOP_n$, 
by Proposition \ref{relJ}, and so $\alpha h\mathscr{J}g_n$ (in $\AOR_n$), since $g_1\mathscr{J}g_n$. 
Hence, $\alpha h= ug_nv$ and $g_n=u'\alpha hv'$, for some $u,v,u',v'\in\AOR_n$, 
from which follows also that $\alpha = ug_nvh$ and so $\alpha\mathscr{J}g_n$. 

Analogously, $\beta\mathscr{J}g_n$ and so $\alpha\mathscr{J}\beta$. 

\smallskip

Finally, we suppose that $n\equiv 2\mod 4$. 

Suppose that $\alpha\in\AOP_n$. Then, by Proposition \ref{relJ}, we have in $\AOP_n$ 
(and so also in $\AOR_n$) that $\alpha\mathscr{J}g_n$, if $\gd(\alpha)$ is even, 
and $\alpha\mathscr{J}g_1$, if $\gd(\alpha)$ is odd. 

Now, suppose that $\alpha\in\AOR_n\setminus\AOP_n$. 
Observe that, in this case, $hg\in\AOR_n$, whence $\alpha hg\in J_{n-1}$. 
Moreover, $\gd(\alpha hg)=\gd(\alpha)$. 
Then, by Proposition \ref{relJ}, in $\AOP_n$, we have $\alpha hg\mathscr{J}g_n$, if $\gd(\alpha)$ is even, 
and $\alpha hg\mathscr{J}g_1$, if $\gd(\alpha)$ is odd. 
If $\alpha hg\mathscr{J}g_n$, then $\alpha hg= ug_nv$ and $g_n=u'\alpha hg v'$, for some $u,v,u',v'\in\AOP_n$, 
from which follows also that $\alpha = ug_nvhg$ and so $\alpha\mathscr{J}g_n$ (in $\AOR_n$). 
Similarly, if $\alpha hg\mathscr{J}g_1$, then $\alpha\mathscr{J}g_1$ (in $\AOR_n$). 
Therefore, we also get $\alpha\mathscr{J}g_n$, if $\gd(\alpha)$ is even, 
and $\alpha\mathscr{J}g_1$, if $\gd(\alpha)$ is odd. 

At this point, by showing that $(g_1,g_n)\not\in\mathscr{J}$, 
which can be done similarly to the proof of Proposition \ref{relJ}, this proof is finished.  
\end{proof}

Let us suppose that $n\not\equiv 2\mod 4$. 
Then, by the previous proposition, $Q_{n-1}$ 
is a $\mathscr{J}$-class of $\AOR_n$. Moreover, 
it is easy to conclude that $Q_{n-1}$ has $n$ $\mathscr{L}$-classes (and $\mathscr{R}$-classes) and its  
$\mathscr{H}$-classes have $n-1$ elements. From (\ref{gnhn}), it follows that the maximal groups of $Q_{n-1}$ are 
dihedral groups of order $n-1\,(=2\frac{n-1}{2})$, if $n$ is odd, and cyclic groups of order $n-1$, if $n\equiv 0\mod 4$. 
Furthermore, we have $Q_0<_\mathscr{J}Q_1<_\mathscr{J}\cdots<_\mathscr{J}Q_{n-2}<_\mathscr{J}Q_{n-1}<_\mathscr{J}Q_n$ in $\AOR_n$. 

\smallskip 

On the other hand, suppose that $n\equiv 2\mod 4$. 
In this case, by Proposition \ref{RrelJ}, $\AOR_n$ has two $\mathscr{J}$-classes of elements of rank $n-1$, namely,
$$
Q_{n-1}^\mathscr{o}=\{\alpha\in Q_{n-1}\mid \mbox{$\gd(\alpha)$ is odd}\}
\quad\text{and}\quad 
Q_{n-1}^\mathscr{e}=\{\alpha\in Q_{n-1}\mid \mbox{$\gd(\alpha)$ is even}\}. 
$$
Moreover, $|Q_{n-1}^\mathscr{o}|=|Q_{n-1}^\mathscr{e}|=\frac{1}{2}n^2(n-1)$,  
both $Q_{n-1}^\mathscr{o}$  and $Q_{n-1}^\mathscr{e}$ have $\frac{n}{2}$ $\mathscr{L}$-classes (and $\mathscr{R}$-classes) and 
their $\mathscr{H}$-classes have size $2n-2$ (in particular, by (\ref{gnhn}), their $\mathscr{H}$-classes of idempotents are dihedral groups of order $2(n-1)$). 
Furthermore, it is easy to check that the poset of $\mathscr{J}$-classes of $\AOR_n$ can be represented by the Hasse diagram of Figure \ref{fAOR}. 
\begin{figure}[ht] 
\centering
\begin{tikzpicture}[scale=0.5]
\draw (0,0) node{\fbox{$Q_n$}} ; 
\draw (-2,-1.5) node{\fbox{$Q_{n-1}^\mathscr{o}$}} ; 
\draw (2,-1.5) node{\fbox{$Q_{n-1}^\mathscr{e}$}} ; 
\draw (0,-3) node{\fbox{$Q_{n-2}$}} ; 
\draw (0,-5) node{\fbox{$Q_1$}} ; 
\draw (0,-6.5) node{\fbox{$Q_0$}} ; 
\draw[thin] (0,-.55) -- (-1,-0.89);
\draw[thin] (0,-.55) -- (1,-0.89);
\draw[thin] (0,-2.42) -- (1,-2.11);
\draw[thin] (0,-2.42) -- (-1,-2.11);
\draw[thin,dotted] (0,-3.65) -- (0,-4.35);
\draw[thin] (0,-5.55) -- (0,-5.95);
\end{tikzpicture}
\caption{The Hasse diagram of $\AOR_n/_\mathscr{J}$, for $n\equiv 2\mod 4$.} \label{fAOR}
\end{figure}

\section{The congruences} \label{congruences} 

Our aim in this section is to describe all congruences of the monoids $\AOP_n$ and $\AOR_n$. 

\medskip 

Let $M$ be a monoid and let $I$ be an ideal of $M$. If $u\in I$ and $a\in M$ are such that $J_a\leqslant_\mathscr{J} J_u$, 
then $a\in MuM$ and so $a\in I$. Therefore, if $M/\mathscr{J} =\{J_{a_0}<_\mathscr{J} J_{a_1}<_\mathscr{J}\cdots<_\mathscr{J} J_{a_m}\}$, 
for some $a_0,a_1,\ldots,a_m\in M$, it is easy to conclude that the ideals of $M$ are the following $m+1$ subsets of $M$: 
$$
I_k=J_{a_0}\cup J_{a_1}\cup\cdots\cup J_{a_k}, ~ 0\leqslant k\leqslant m. 
$$
On the other hand, if 
$M/\mathscr{J} =\{J_{a_0}<_\mathscr{J} J_{a_1}<_\mathscr{J}\cdots<_\mathscr{J} J_{a_{m-2}}<_\mathscr{J} J_{a'_{m-1}}, J_{a''_{m-1}}<_\mathscr{J} J_{a_{m}}\}$ 
(with $J_{a'_{m-1}}$ and $J_{a''_{m-1}}$ $\leqslant_\mathscr{J}$-incomparable $\mathscr{J}$-classes), 
for some $a_0,a_1,\ldots,a_{m-2},a'_{m-1},a''_{m-1},a_m\in M$, it is not difficult to deduce that the ideals of $M$ are the following $m+3$ subsets of $M$: 
$$
I_k=J_{a_0}\cup J_{a_1}\cup\cdots\cup J_{a_k}, ~ 0\leqslant k\leqslant m-2, 
$$ 
$I'_{m-1}=I_{m-2}\cup J_{a'_{m-1}}$, $I''_{m-1}=I_{m-2}\cup J_{a''_{m-1}}$, $I_{m-1}=I_{m-2}\cup J_{a'_{m-1}}\cup J_{a''_{m-1}}$ and $I_m=M$. 

\medskip 

Now, recall that 
$$
\AO_n/\mathscr{J}=\left\{J_0^{\AO_n}<_\mathscr{J}J_1^{\AO_n}<_\mathscr{J}\cdots<_\mathscr{J}J_{n-2}^{\AO_n}<_\mathscr{J} 
J_{n-1}^{\mathscr{o},\AO_n}, J_{n-1}^{\mathscr{e},\AO_n} <_\mathscr{J}J_n^{\AO_n}\right\}
$$ 
(with $J_{n-1,\mathscr{o}}^{\AO_n}$ and $J_{n-1,\mathscr{e}}^{\AO_n}$ $\leqslant_\mathscr{J}$-incomparable $\mathscr{J}$-classes), where 
$$
J_k^{\AO_n}=\left\{\alpha\in\POI_n\mid |\im(\alpha)|=k\right\}, ~ 0\leqslant k\leqslant n-2, 
$$
$$
J_{n-1}^{\mathscr{o},\AO_n}=\left\{\alpha\in\POI_n\mid |\im(\alpha)|=n-1 ~\text{and}~ (-1)^{\gd(\alpha)}=(-1)^{\gi(\alpha)}=-1\right\}, 
$$
$$
J_{n-1}^{\mathscr{e},\AO_n}=\left\{\alpha\in\POI_n\mid |\im(\alpha)|=n-1 ~\text{and}~ (-1)^{\gd(\alpha)}=(-1)^{\gi(\alpha)}=1\right\} 
$$
and $J_n^{\AO_n}=\{\id_n\}$. 
Consequently, $\AO_n$ has the following $n+3$ ideals:
$$
I_k^{\AO_n}=\left\{\alpha\in\AO_n\mid |\im(\alpha)|\leqslant k\right\}, ~\text{with}~ 0\leqslant k\leqslant n, 
$$
$I^{\mathscr{o},\AO_n}_{n-1}=I_{n-2}^{\AO_n}\cup J_{n-1}^{\mathscr{o},\AO_n}$ and 
$I^{\mathscr{e},\AO_n}_{n-1}=I_{n-2}^{\AO_n}\cup J_{n-1}^{\mathscr{e},\AO_n}$.  
For more details, see \cite{Fernandes:2024sub}. 

\smallskip 

The following result was proved by Fernandes in \cite{Fernandes:2024sub}. 

\begin{theorem}[{\cite[Theorem 3.3]{Fernandes:2024sub}}] \label{conAO} 
The congruences of the monoid $\AO_n$ are exactly its $n+3$ Rees congruences. 
\end{theorem} 

\medskip 

Let $M$ be a finite monoid and let $J$ be a $\mathscr{J}$-class
of $M$. Define 
$$
A(J)=\left\{ a\in M\mid J_a<_\mathscr{J}J \right\}
\quad\text{and}\quad 
B(J)=\left\{ a\in M\mid J\not\leqslant_\mathscr{J}J_a \right\}.
$$
It is a routine matter to show that $A(J)$ and $B(J)$ 
are ideals of $M$ such that $A(J)\subseteq B(J)$ 
(in fact, more precisely, 
$
B(J)=A(J)\cup \left\{ a\in M\mid \mbox{$J$ and $J_a$ are $\leqslant_\mathscr{J}$-incomparable}\right\}
$) and $J\cap A(J)=J\cap B(J)=\emptyset$. 

Now, suppose that $J$ is a regular $\mathscr{J}$-class of $M$ and let $H$ be a group $\mathscr{H}$-class of $M$ contained in $J$. 
Suppose there exists a mapping $\widetilde{\cdot}:J\longrightarrow H$, $x\longmapsto \widetilde{x}$, such that 
\begin{equation}\label{fundcon} 
\mbox{$xy\in J$ implies $\widetilde{xy}=\widetilde{x}\widetilde{y}$, for all $x,y\in J$.} 
\end{equation}
Let $\varrho$ be a congruence of (the group) $H$. 
Let $\contheta{J}^\varrho$ [respectively, $\conpi{J}^\varrho$] be the binary relation on $M$ defined by:
for $x,y\in M$, $x\contheta{J}^\varrho y$ [respectively, $x\conpi{J}^\varrho y$] if and only if
\begin{enumerate}
\item $x=y$; or
\item $x,y\in A(J)$ [respectively, $x,y\in B(J)$]; or
\item $x,y\in J$, $x\mathscr{H}y$ and $\widetilde{x}\varrho\widetilde{y}$.
\end{enumerate}
Notice that, clearly, $\contheta{J}^\varrho\subseteq\conpi{J}^\varrho$. 
Obviously, if $A(J)=B(J)$, then $\contheta{J}^\varrho=\conpi{J}^\varrho$. 

Under these conditions (in fact, even a little more general), Fernandes et al. proved in \cite[Theorem 2.3]{Fernandes&Gomes&Jesus:2009} 
that $\conpi{J}^\varrho$ is a congruence of $M$. With a proof analogous to that of \cite[Theorem 2.3]{Fernandes&Gomes&Jesus:2009}, we can also show that: 

\begin{proposition}\label{contheta} 
Under the above conditions, $\contheta{J}^\varrho$ is a congruence of $M$ such that $\contheta{J}^\varrho\subseteq\conpi{J}^\varrho$. 
\end{proposition} 

Observe that, if $J$ consists of a single $\mathscr{H}$-class (in particular, if $J$ is the group of units of $M$), the identity application (or any group homomorphism) of $J$ verifies (\ref{fundcon}). Therefore, in this case, if $\varrho$ is a congruence of (the group) $J$, 
the binary relation $\contheta{J}^\varrho$ [respectively, $\conpi{J}^\varrho$] defined on $M$ by 
 $x\contheta{J}^\varrho y$ [respectively, $x\conpi{J}^\varrho y$] if and only if
\begin{enumerate}
\item $x=y$; or
\item $x,y\in A(J)$ [respectively, $x,y\in B(J)$]; or
\item $x,y\in J$ and $x\varrho y$, 
\end{enumerate}
for all $x,y\in M$, is a congruence of $M$. 
In particular, if $J$ is the group of units of $M$, we obtain $\conpi{J}^\varrho\,(=\contheta{J}^\varrho)$ defined by 
$x\conpi{J}^\varrho y$ if and only if either $x,y\in J$ and $x\varrho y$ or $x,y\in M\setminus J$, for all $x,y\in M$. 

\smallskip 

Next, consider two regular $\mathscr{J}$-classes $J$ and $J'$ of $M$ such that $A(J)=A(J')$. 
Let $H$ and $H'$ be group $\mathscr{H}$-classes of $M$ such that $H\subseteq J$ and $H'\subseteq J'$. 
Let $\varrho$ and $\varrho'$ be congruences of (the groups) $H$ and $H'$, respectively. 
Suppose there exist mappings $\widetilde{\cdot}:J\longrightarrow H$, $x\longmapsto \widetilde{x}$, and 
$\widetilde{\cdot}:J'\longrightarrow H'$, $x\longmapsto \widetilde{x}$, verifying (\ref{fundcon}) and 
consider the related congruences $\contheta{J}^\varrho$ and $\contheta{J'}^{\varrho'}$ of $M$. 
Under these conditions, it is a routine matter to check that $\contheta{J}^\varrho\cup\contheta{J'}^{\varrho'}$ is also a transitive binary relation. 
Thus, we immediately have the following result.   

\begin{proposition}\label{cupcon} 
Under the previous conditions, $\contheta{J}^\varrho\cup\contheta{J'}^{\varrho'}$ is a congruence of $M$. 
\end{proposition} 

\medskip 

Now, suppose that $M$ is an inverse submonoid of $\I_n$. 

Let $1\leqslant k\leqslant n$ and let $J$ be a $\mathscr{J}$-class of $M$ such that $J\subseteq J_k^{\I_n}$.  
Let $X$ be a subset of $\Omega_n$ such that $|X|=k$ and let $\varepsilon=\id_X$. 
Suppose two mappings $\widetilde{\cdot}_L:J\longrightarrow J$, $\alpha\longmapsto \widetilde{\alpha}_L$, and 
 $\widetilde{\cdot}_R:J\longrightarrow J$, $\alpha\longmapsto \widetilde{\alpha}_R$, are defined  such that, 
 for $\alpha\in J$, 
 $$
  \widetilde{\alpha}_L=\transf{X\\ \dom(\alpha)}
  \quad\text{and}\quad 
  \widetilde{\alpha}_R=\transf{\im(\alpha)\\ X}. 
 $$
Then, we have:
\begin{enumerate}
\item $\varepsilon=\widetilde{\alpha}_L\widetilde{\alpha}_L^{-1}=\widetilde{\alpha}_R^{-1}\widetilde{\alpha}_R\in J$, 
$\widetilde{\alpha}_L^{-1}\widetilde{\alpha}_L=\alpha\alpha^{-1}$ and $\widetilde{\alpha}_R\widetilde{\alpha}_R^{-1}=\alpha^{-1}\alpha$, 
for $\alpha\in J$; 
\item A mapping $\widetilde{\cdot}:J\longrightarrow H_\varepsilon^M$, $\alpha\longmapsto\widetilde{\alpha}$, 
defined by $\widetilde{\alpha}=\widetilde{\alpha}_L\alpha\widetilde{\alpha}_R$, for $\alpha\in J$. 
\end{enumerate} 

In addition, suppose that the mappings  $\widetilde{\cdot}_L$ and $\widetilde{\cdot}_R$ verify the following properties: 
\begin{enumerate}
\item[$\widetilde{1}$.] For $\alpha,\beta\in J$, if $\alpha\mathscr{R}\beta$, then $\widetilde{\alpha}_L=\widetilde{\beta}_L$; 
\item[$\widetilde{2}$.] For $\alpha,\beta\in J$, if $\alpha\mathscr{L}\beta$, then $\widetilde{\alpha}_R=\widetilde{\beta}_R$; 
\item[$\widetilde{3}$.] For $\alpha\in J$,  $\widetilde{(\alpha^{-1})}_R=\widetilde{\alpha}_L^{-1}$.  
\end{enumerate}
Hence, 
$$
\mbox{$\alpha\beta\in J$ implies $\widetilde{\alpha\beta}=\widetilde{\alpha}\widetilde{\beta}$, for all $\alpha,\beta\in J$.} 
$$
In fact, if $\alpha,\beta\in J$ are such that $\alpha\beta\in J$, then $\im(\alpha)=\dom(\beta)$, i.e. $\alpha\mathscr{L}\beta^{-1}$, 
and $\alpha\mathscr{R}\alpha\beta\mathscr{L}\beta$, whence 
$\widetilde{(\alpha\beta)}_L=\widetilde{\alpha}_L$, $\widetilde{(\alpha\beta)}_R=\widetilde{\beta}_R$ and 
$\widetilde{\alpha}_R=\widetilde{(\beta^{-1})}_R=\widetilde{\beta}_L^{-1}$, from which follows 
$$
\widetilde{\alpha\beta}=\widetilde{(\alpha\beta)}_L\alpha\beta\widetilde{(\alpha\beta)}_R = 
\widetilde{\alpha}_L\alpha(\alpha^{-1}\alpha)\beta\widetilde{\beta}_R =
\widetilde{\alpha}_L\alpha(\widetilde{\alpha}_R\widetilde{\alpha}_R^{-1})\beta\widetilde{\beta}_R =
\widetilde{\alpha}_L\alpha\widetilde{\alpha}_R\widetilde{\beta}_L\beta\widetilde{\beta}_R =\widetilde{\alpha}\widetilde{\beta}. 
$$
Moreover, it is a routine matter to check that, for $\alpha\in J$, 
the mapping $H_\alpha^M\longrightarrow H_\varepsilon^M$, $\gamma\longmapsto\widetilde{\gamma}$, is a bijection with inverse 
$H_\varepsilon^M\longrightarrow H_\alpha^M$, $\sigma\longmapsto\widetilde{\alpha}_L^{-1}\sigma\widetilde{\alpha}_R^{-1}$. 

\medskip 

From now on, suppose that $M\in\{\AOP_n,\AOR_n\}$. 

\smallskip 

Let $1\leqslant k\leqslant n-2$ and consider the  $\mathscr{J}$-class $J=J_k^{\I_n}\cap M$ of $M$.  
Observe that, $J_k^{\POI_n}\subseteq J$. 

Let us define mappings $\widetilde{\cdot}_L:J\longrightarrow J$, $\alpha\longmapsto \widetilde{\alpha}_L$, and 
 $\widetilde{\cdot}_R:J\longrightarrow J$, $\alpha\longmapsto \widetilde{\alpha}_R$, by 
 $$
  \widetilde{\alpha}_L=\transf{\{1,\ldots,k\}\\ \dom(\alpha)}\in\POI_n 
  \quad\text{and}\quad 
  \widetilde{\alpha}_R=\transf{\im(\alpha)\\ \{1,\ldots,k\}}\in\POI_n, 
 $$
 for all $\alpha\in J$. Let $\varepsilon=\id_{\{1,\ldots,k\}}$ and 
 consider also the mapping $\widetilde{\cdot}:J\longrightarrow H_\varepsilon^M$, $\alpha\longmapsto\widetilde{\alpha}$, 
defined by $\widetilde{\alpha}=\widetilde{\alpha}_L\alpha\widetilde{\alpha}_R$, for $\alpha\in J$. 
 It is clear that $\widetilde{\cdot}_L$ and $\widetilde{\cdot}_R$ 
 verify $\widetilde{1}$, $\widetilde{2}$ and $\widetilde{3}$.  

\smallskip 

Next, let $J=J_{n-1}^\mathscr{e}$, with $M=\AOP_n$ and $n$ even, or $J=Q_{n-1}^\mathscr{e}$, with $M=\AOR_n$ and $n\equiv 2\mod 4$.  
Then, for each $\alpha\in J$, we have $\alpha_L,\alpha_R\in J$. Hence, we can consider 
mappings $\widetilde{\cdot}_L:J\longrightarrow J$, $\alpha\longmapsto \widetilde{\alpha}_L$, and 
 $\widetilde{\cdot}_R:J\longrightarrow J$, $\alpha\longmapsto \widetilde{\alpha}_R$, defined by 
 $\widetilde{\alpha}_L=\alpha_L$ and  $\widetilde{\alpha}_R=\alpha_R$,  for all $\alpha\in J$. 
 As in the previous case, it is clear that $\widetilde{\cdot}_L$ and $\widetilde{\cdot}_R$ 
 verify $\widetilde{1}$, $\widetilde{2}$ and $\widetilde{3}$. 
So, let $\varepsilon=\id_{\{1,\ldots,n-1\}}$ and 
$\widetilde{\cdot}:J\longrightarrow H_\varepsilon^M$, $\alpha\longmapsto\widetilde{\alpha}$, 
be the mapping  defined by $\widetilde{\alpha}=\widetilde{\alpha}_L\alpha\widetilde{\alpha}_R$, for $\alpha\in J$. 

\smallskip 

Now, let $J=J_{n-1}^\mathscr{o}$, with $M=\AOP_n$ and $n$ even, or $J=Q_{n-1}^\mathscr{o}$, with $M=\AOR_n$ and $n\equiv 2\mod 4$.  
For each $\alpha\in J$, define 
 $$
  \widetilde{\alpha}_L=\transf{\{2,\ldots,n\}\\ \dom(\alpha)}\in\POI_n 
  \quad\text{and}\quad 
  \widetilde{\alpha}_R=\transf{\im(\alpha)\\ \{2,\ldots,n\}}\in\POI_n. 
 $$
Clearly, $\widetilde{\alpha}_L, \widetilde{\alpha}_R\in J$, for $\alpha\in J$. 
Hence, we have well-defined mappings $\widetilde{\cdot}_L:J\longrightarrow J$, $\alpha\longmapsto \widetilde{\alpha}_L$, 
 $\widetilde{\cdot}_R:J\longrightarrow J$, $\alpha\longmapsto \widetilde{\alpha}_R$, and 
 $\widetilde{\cdot}:J\longrightarrow H_\varepsilon^M$, $\alpha\longmapsto\widetilde{\alpha}$, 
 defined by $\widetilde{\alpha}=\widetilde{\alpha}_L\alpha\widetilde{\alpha}_R$, for $\alpha\in J$, with $\varepsilon=\id_{\{2,\ldots,n\}}$. 
It is also clear that $\widetilde{\cdot}_L$ and $\widetilde{\cdot}_R$ 
 verify $\widetilde{1}$, $\widetilde{2}$ and $\widetilde{3}$.
 
\smallskip

In our next case, for an odd $n$, we take $J=J_{n-1}$, if $M=\AOP_n$, and $J=Q_{n-1}$, if $M=\AOR_n$. 
For each $\alpha\in J$, define 
$$
 \widetilde{\alpha}_L = \left\{\begin{array}{ll}
 \alpha_L & \mbox{if $\gd(\alpha)$ is odd}\\
 g_n^{-1}\alpha_L & \mbox{if $\gd(\alpha)$ is even}
 \end{array}\right. 
 \quad\text{and}\quad
  \widetilde{\alpha}_R = \left\{\begin{array}{ll}
 \alpha_R & \mbox{if $\gi(\alpha)$ is odd}\\
 \alpha_R g_n & \mbox{if $\gi(\alpha)$ is even.}
 \end{array}\right. 
$$

Let $\alpha\in J$. Then, we have 
$\alpha_L\in\AO_n$ [respectively, $\alpha_R\in\AO_n$] if and only if $\gd(\alpha)$ [respectively, $\gi(\alpha)$] is odd. 
On the other hand, as $g_n\not\in\AI_n$, if $\gd(\alpha)$ [respectively, $\gi(\alpha)$] is even, then 
$\overline{g_n^{-1}\alpha_L}=\overline{g_n^{-1}}\overline{\alpha_L}\in\A_n$ 
[respectively, $\overline{\alpha_R g_n}=\overline{\alpha_R}\,\overline{g_n}\in\A_n$] and so 
$g_n^{-1}\alpha_L\in\AOP_n$ [respectively,  $\alpha_R g_n \in\AOP_n$]. 
Therefore, $\widetilde{\alpha}_L ,\widetilde{\alpha}_R\in J$. 

Hence, we have well-defined mappings $\widetilde{\cdot}_L:J\longrightarrow J$, $\alpha\longmapsto \widetilde{\alpha}_L$, 
 $\widetilde{\cdot}_R:J\longrightarrow J$, $\alpha\longmapsto \widetilde{\alpha}_R$, and 
 $\widetilde{\cdot}:J\longrightarrow H_\varepsilon^M$, $\alpha\longmapsto\widetilde{\alpha}$, 
 defined by $\widetilde{\alpha}=\widetilde{\alpha}_L\alpha\widetilde{\alpha}_R$, for $\alpha\in J$, with $\varepsilon=\id_{\{1,\ldots,n-1\}}$. 

Clearly, $\widetilde{\cdot}_L$ and $\widetilde{\cdot}_R$ verify $\widetilde{1}$ and $\widetilde{2}$. 
On the other hand, for $\alpha\in J$, we have 
$$
 \widetilde{(\alpha^{-1})}_R = \left\{\begin{array}{ll}
 (\alpha^{-1})_R & \mbox{if $\gd(\alpha)=\gi(\alpha^{-1})$ is odd}\\
 (\alpha^{-1})_R\, g_n & \mbox{if $\gd(\alpha)=\gi(\alpha^{-1})$ is even}
 \end{array}\right. 
 =\left\{\begin{array}{ll}
 \alpha_L^{-1} & \mbox{if $\gd(\alpha)$ is odd}\\
 \alpha^{-1}_L g_n & \mbox{if $\gd(\alpha)$ is even}
 \end{array}\right. 
=  \widetilde{\alpha}_L^{-1}, 
$$
whence $\widetilde{\cdot}_L$ and $\widetilde{\cdot}_R$ also verify $\widetilde{3}$. 

\smallskip 

Finally, let us take $J=Q_{n-1}$, with $M=\AOR_n$ and $n\equiv 0\mod 4$.  Let $\alpha\in J$. 
In this case, $\alpha_L\in\AO_n$ [respectively, $\alpha_R\in\AO_n$] if and only if $\gd(\alpha)$ [respectively, $\gi(\alpha)$] is even. 
Moreover, as $h_n\not\in\AI_n$,  
if $\gd(\alpha)$ is odd, then 
$\overline{h_n\alpha_L}=\overline{h_n}\overline{\alpha_L}\in\A_n$ and so 
$h_n\alpha_L\in\AOR_n$. Similarly, if $\gi(\alpha)$ is odd, then $\alpha_Rh_n\in\AOR_n$. 

Therefore, we can define mappings $\widetilde{\cdot}_L:J\longrightarrow J$, $\alpha\longmapsto \widetilde{\alpha}_L$, 
 $\widetilde{\cdot}_R:J\longrightarrow J$, $\alpha\longmapsto \widetilde{\alpha}_R$, by 
 $$
 \widetilde{\alpha}_L = \left\{\begin{array}{ll}
 \alpha_L & \mbox{if $\gd(\alpha)$ is even}\\
 h_n\alpha_L & \mbox{if $\gd(\alpha)$ is odd}
 \end{array}\right. 
 \quad\text{and}\quad
  \widetilde{\alpha}_R = \left\{\begin{array}{ll}
 \alpha_R & \mbox{if $\gi(\alpha)$ is even}\\
 \alpha_R h_n & \mbox{if $\gi(\alpha)$ is odd,}
 \end{array}\right. 
$$
for $\alpha\in J$. 
It is clear that $\widetilde{\cdot}_L$ and $\widetilde{\cdot}_R$ verify $\widetilde{1}$ and $\widetilde{2}$. 
In addition, if $\alpha\in J$, then  
$$
 \widetilde{(\alpha^{-1})}_R = \left\{\begin{array}{ll}
 (\alpha^{-1})_R & \mbox{if $\gd(\alpha)=\gi(\alpha^{-1})$ is even}\\
 (\alpha^{-1})_R\, h_n & \mbox{if $\gd(\alpha)=\gi(\alpha^{-1})$ is odd}
 \end{array}\right. 
 =\left\{\begin{array}{ll}
 \alpha_L^{-1} & \mbox{if $\gd(\alpha)$ is even}\\
 \alpha^{-1}_L h_n =(h_n\alpha_L)^{-1}& \mbox{if $\gd(\alpha)$ is odd}
 \end{array}\right. 
=  \widetilde{\alpha}_L^{-1}. 
$$
Hence, $\widetilde{\cdot}_L$ and $\widetilde{\cdot}_R$ also verify $\widetilde{3}$. 
With $\varepsilon=\id_{\{1,\ldots,n-1\}}$, let us also consider the mapping $\widetilde{\cdot}:J\longrightarrow H_\varepsilon^M$, 
$\alpha\longmapsto\widetilde{\alpha}$, defined by $\widetilde{\alpha}=\widetilde{\alpha}_L\alpha\widetilde{\alpha}_R$, for $\alpha\in J$. 

\smallskip 

Let $J$ be any of the $\mathscr{J}$-classes of $M$ considered above.  
Since the mappings $\widetilde{\cdot}_L$ and $\widetilde{\cdot}_R$ verify $\widetilde{1}$, $\widetilde{2}$ and $\widetilde{3}$, 
we get $\widetilde{\alpha\beta}=\widetilde{\alpha}\widetilde{\beta}$, for all $\alpha,\beta\in J$ such that $\alpha\beta\in J$. 
Moreover, for $\alpha\in J$, 
the mappings $H_\alpha^M\longrightarrow H_\varepsilon^M$, 
$\gamma\longmapsto\widetilde{\gamma}=\widetilde{\alpha}_L\gamma\widetilde{\alpha}_R$, and 
$H_\varepsilon^M\longrightarrow H_\alpha^M$, $\sigma\longmapsto\widetilde{\alpha}_L^{-1}\sigma\widetilde{\alpha}_R^{-1}$, 
are mutually inverse bijections. 
In what follows, for each group congruence $\varrho$ of $H_\varepsilon^M$, 
we consider the congruences $\contheta{J}^\varrho$ and $\conpi{J}^\varrho$ of $M$ associated to the (fixed) mapping 
$\widetilde{\cdot}:J\longrightarrow H_\varepsilon^M$, $\alpha\longmapsto\widetilde{\alpha}$. 
Notice that, except for $J\in \{J_{n-1}^\mathscr{o},J_{n-1}^\mathscr{e}\}$, with $M=\AOP_n$ and $n$ even, 
and $J\in \{Q_{n-1}^\mathscr{o},Q_{n-1}^\mathscr{e}\}$, with $M=\AOR_n$ and $n\equiv 2\mod 4$, 
we have  $\contheta{J}^\varrho=\conpi{J}^\varrho$. 

Furthermore, 
if $J$ is the group of units of $M$, in the following, for any (group) congruence $\varrho$ of $J$, 
we consider the congruence $\conpi{J}^\varrho\,(=\contheta{J}^\varrho)$ of $M$ defined by 
$\alpha\conpi{J}^\varrho \beta$ if and only if either $\alpha,\beta\in J$ and $\alpha\varrho \beta$ or $\alpha,\beta\in M\setminus J$, 
for all $\alpha,\beta\in M$. 

\medskip 

We introduce the following notation in order to deal with similar cases simultaneously.  Let 
$$
J_k^M=\left\{\begin{array}{ll}
J_k &\mbox{if $M=\AOP_n$} \\
Q_k &\mbox{if $M=\AOR_n$,} 
\end{array}\right. 
$$
for $0\leqslant k\leqslant n-2$ and $k=n$, 
$$
J_{n-1}^M=\left\{\begin{array}{ll}
J_{n-1} &\mbox{if $M=\AOP_n$ and $n$ is odd} \\
Q_{n-1} &\mbox{if $M=\AOR_n$ and $n\not\equiv 2\mod 4$,} 
\end{array}\right. 
$$
$$
J_{n-1}^{\mathscr{o},M}=\left\{\begin{array}{ll}
J_{n-1}^\mathscr{o} &\mbox{if $M=\AOP_n$ and $n$ is even} \\
Q_{n-1}^\mathscr{o} &\mbox{if $M=\AOR_n$ and $n\equiv 2\mod 4$} 
\end{array}\right. 
$$
and
$$
J_{n-1}^{\mathscr{e},M}=\left\{\begin{array}{ll}
J_{n-1}^\mathscr{e} &\mbox{if $M=\AOP_n$ and $n$ is even} \\
Q_{n-1}^\mathscr{e} &\mbox{if $M=\AOR_n$ and $n\equiv 2\mod 4$.} 
\end{array}\right. 
$$

Now, for $0\leqslant k\leqslant n$, let 
$$
I_k^M=\left\{\alpha\in M\mid |\im(\alpha)|\leqslant k\right\}. 
$$
Then, bearing in mind the observation made at the beginning of this section, 
$I_k^M$ is an ideal of $M$, for $0\leqslant k\leqslant n$. 
Moreover, except for $M=\AOP_n$, with $n$ even, and for $M=\AOR_n$, with $n\equiv 2\mod 4$, 
$I_0^M, I_1^M,\ldots, I_n^M$ are all the ideals of $M$. On the other hand, 
if either $M=\AOP_n$ and $n$ is even or $M=\AOR_n$ and $n\equiv 2\mod 4$, then $M$ has more two ideals, namely 
$$
I^{\mathscr{o},M}_{n-1}=I_{n-2}^M\cup J_{n-1}^{\mathscr{o},M} ~\text{and}~ I^{\mathscr{e},M}_{n-1}=I_{n-2}^M\cup J_{n-1}^{\mathscr{e},M}. 
$$

The following lemmas will be useful to prove the main result of this section. 

\begin{lemma}\label{pori}
Let $\alpha,\beta\in\PORI_n$ be such that $|\im(\alpha)|\geqslant2$, $\alpha\neq\beta$ and $\alpha\mathscr{H}\beta$. 
Let $\varepsilon_{\mathrm{min}}=\id_{\dom(\alpha)\setminus\{\min\dom(\alpha)\}}$ and 
$\varepsilon_{\mathrm{max}}=\id_{\dom(\alpha)\setminus\{\max\dom(\alpha)\}}$. 
Then, $\im(\varepsilon_{\mathrm{min}}\alpha)\neq\im(\varepsilon_{\mathrm{min}}\beta)$ or 
$\im(\varepsilon_{\mathrm{max}}\alpha)\neq\im(\varepsilon_{\mathrm{max}}\beta)$. 
\end{lemma}
\begin{proof}
Let us take $\dom(\alpha)=\{a_1<\cdots<a_p\}$ and $\im(\alpha)=\{b_1<\cdots<b_p\}$ ($2\leqslant p\leqslant n$). 
Since $\alpha\mathscr{H}\beta$, we have $\dom(\beta)=\dom(\alpha)$ and $\im(\beta)=\im(\alpha)$. 
If $p=2$, then it is clear that $\im(\varepsilon_{\mathrm{min}}\alpha)\neq\im(\varepsilon_{\mathrm{min}}\beta)$ and
$\im(\varepsilon_{\mathrm{max}}\alpha)\neq\im(\varepsilon_{\mathrm{max}}\beta)$, since $\alpha\neq\beta$.
So, let $p\geqslant3$ and suppose that $\im(\varepsilon_{\mathrm{min}}\alpha)=\im(\varepsilon_{\mathrm{min}}\beta)$.  
Then, $a_1\alpha=a_1\beta=b_i$, for some $1\leqslant i\leqslant p$. 
Since $\alpha\neq\beta$, we must have $\{a_p\alpha,a_p\beta\}=\{b_{i-1},b_{i+1}\}$, with $b_0=b_p$ and $b_{p+1}=b_1$. 
Since $p\geqslant3$, we have $b_{i-1}\neq b_{i+1}$ and so 
$\im(\varepsilon_{\mathrm{max}}\alpha)\neq\im(\varepsilon_{\mathrm{max}}\beta)$, as required. 
\end{proof}

Observe that, for $n\geqslant2$, $0\leqslant k\leqslant n-1$ and $1\leqslant i\leqslant n$, 
it is clear that 
\begin{equation}\label{g_n}
\mbox{$\im(\id_{\Omega_n\setminus\{i\}}g^k)=\id_{\Omega_n\setminus\{i\}}$ if and only if $k=0$.}
\end{equation}
On the other hand, for permutations of the dihedral group, we have: 

\begin{lemma}\label{d_2n}
Let $n\geqslant5$ and let $\sigma\in\D_{2n}$ be such that $\sigma\neq\id_n$. Then: 
\begin{enumerate}
\item $\im(\id_{\Omega_n\setminus\{1\}}\sigma)\neq \Omega_n\setminus\{1\}$ or 
$\im(\id_{\Omega_n\setminus\{3\}}\sigma)\neq \Omega_n\setminus\{3\}$;  
\item $\im(\id_{\Omega_n\setminus\{2\}}\sigma)\neq \Omega_n\setminus\{2\}$ or 
$\im(\id_{\Omega_n\setminus\{n\}}\sigma)\neq \Omega_n\setminus\{n\}$. 
\end{enumerate}
\end{lemma}
\begin{proof}
If $\im(\id_{\Omega_n\setminus\{1\}}\sigma) = \Omega_n\setminus\{1\}$, then $(1)\sigma=1$, 
whence $\sigma=hg$, since $\sigma\neq\id_n$, and so $3\sigma=n-1\geqslant4$, 
which implies that $\im(\id_{\Omega_n\setminus\{3\}}\sigma)\neq \Omega_n\setminus\{3\}$. Similarly, 
if $\im(\id_{\Omega_n\setminus\{2\}}\sigma) = \Omega_n\setminus\{2\}$, then $(2)\sigma=2$, 
whence $\sigma=hg^{3}$, since $\sigma\neq\id_n$, and so $n\sigma=4<n$, 
which implies that $\im(\id_{\Omega_n\setminus\{n\}}\sigma)\neq \Omega_n\setminus\{n\}$, 
as required. 
\end{proof}

We are now in a position to present and demonstrate our description of the congruences of the monoids $\AOP_n$ and $\AOR_n$. 

\begin{theorem}\label{conM}
Let $M\in\{\AOP_n,\AOR_n\}$. Then, $\con(M)$ consists of the universal congruence $\omega$ and all congruences of the form: 
\begin{enumerate}
\item $\conpi{J_k^M}^\varrho$, with $\varrho$ a congruence of the group $H^M_{\id_{\{1,\ldots,k\}}}$, 
for some $1\leqslant k\leqslant n$, 
if either $M=\AOP_n$ and $n$ is odd or $M=\AOR_n$ and $n\not\equiv 2\mod 4$;
 \item $\conpi{J_k^M}^\varrho$, with $\varrho$ a congruence of the group $H^M_{\id_{\{1,\ldots,k\}}}$, 
 for some $1\leqslant k\leqslant n-2$ or $k=n$, 
and  $\conpi{J_{n-1}^{\mathscr{o},M}}^{\varrho_1}$, $\conpi{J_{n-1}^{\mathscr{e},M}}^{\varrho_2}$ and 
 $\contheta{J_{n-1}^{\mathscr{o},M}}^{\varrho_1}\cup \contheta{J_{n-1}^{\mathscr{e},M}}^{\varrho_2}$, 
 with $\varrho_1$ and $\varrho_2$ congruences of the groups $H^M_{\id_{\{2,\ldots,n\}}}$ and $H^M_{\id_{\{1,\ldots,n-1\}}}$, respectively, 
 if either $M=\AOP_n$ and $n$ is even or $M=\AOR_n$ and $n\equiv 2\mod 4$. 
\end{enumerate}
\end{theorem}
\begin{proof}
Let $\rho$ be a congruence of $M$. Then, $\rho'=\rho\cap (\AO_n\times\AO_n)$ is a congruence of the monoid $\AO_n$ and so, 
by Theorem \ref{conAO}, $\rho'=\sim_{I'}$, for some ideal $I'$ of $\AO_n$. 
Let 
\begin{description}
\item $I=I_k^M$, if $I'=I_k^{\AO_n}$, for some $0\leqslant k\leqslant n$, 
\item $I=I_{n-1}^M$, if $I'=I^{\mathscr{o},\AO_n}_{n-1}$ or $I'=I^{\mathscr{e},\AO_n}_{n-1}$, 
with either $M=\AOP_n$ and $n$ odd or $M=\AOR_n$ and $n\not\equiv 2\mod 4$, 
\item $I=I^{\mathscr{o},M}_{n-1}$, if $I'=I^{\mathscr{o},\AO_n}_{n-1}$, 
with either $M=\AOP_n$ and $n$ even or $M=\AOR_n$ and $n\equiv 2\mod 4$, and 
\item $I=I^{\mathscr{e},M}_{n-1}$, if $I'=I^{\mathscr{e},\AO_n}_{n-1}$, 
with either $M=\AOP_n$ and $n$ even or $M=\AOR_n$ and $n\equiv 2\mod 4$. 
\end{description}
Observe that, $I'\subseteq I$. 

\smallskip 

Let $\alpha\in M$ be such that $\alpha\mathscr{J}\beta$ (in $M$), for some $\beta\in I'$. 
Then, $\alpha=\gamma\beta\lambda$, for some $\gamma,\lambda\in M$, and $\beta\rho'\emptyset$. 
Hence, $\beta\rho\emptyset$ and so $\alpha=\gamma\beta\lambda\rho\gamma\emptyset\lambda=\emptyset$. 
Thus, $\sim_I\subseteq\rho$. 

Notice that, if $I=I_n^M$, then we can already conclude that $\rho=\omega$. 

\smallskip 

Next, let $\alpha\in M\setminus I$. Let $\beta\in M$ be such that $\alpha\rho\beta$. 
Then, as $M$ is an inverse semigroup, we also have $\alpha^{-1}\rho\beta^{-1}$, 
whence $\alpha\alpha^{-1}\rho\beta\beta^{-1}$ and $\alpha^{-1}\alpha\rho\beta^{-1}\beta$. 
Therefore, $\alpha\alpha^{-1}\rho'\beta\beta^{-1}$ and $\alpha^{-1}\alpha\rho'\beta^{-1}\beta$. 
On the other hand, as $\alpha\not\in I$ and $I'\subseteq I$, we get $\alpha\alpha^{-1},\alpha^{-1}\alpha\not\in I'$. 
Hence, $\alpha\alpha^{-1}=\beta\beta^{-1}$ and $\alpha^{-1}\alpha=\beta^{-1}\beta$, i.e. $\alpha\mathscr{H}\beta$. 

\smallskip 

Let $\alpha\in M$ be such that $1\leqslant |\im(\alpha)|\leqslant n-1$. 
Consider the mutually inverse bijections $H_\alpha^M\longrightarrow H_\varepsilon^M$, 
$\gamma\longmapsto\widetilde{\gamma}=\widetilde{\alpha}_L\gamma\widetilde{\alpha}_R$, and 
$H_\varepsilon^M\longrightarrow H_\alpha^M$, $\sigma\longmapsto\widetilde{\alpha}_L^{-1}\sigma\widetilde{\alpha}_R^{-1}$, 
set out above. 
Let $\varrho=\rho\cap(H_\varepsilon^M\times H_\varepsilon^M)$. 
Let $\beta\in M$ be such that $\alpha\mathscr{H}\beta$.
Then, 
$$
\alpha\rho\beta \quad \Longrightarrow \quad 
\widetilde{\alpha}=\widetilde{\alpha}_L\alpha\widetilde{\alpha}_R \rho \widetilde{\alpha}_L\beta\widetilde{\alpha}_R=\widetilde{\beta} 
\quad \Longrightarrow \quad \alpha = \widetilde{\alpha}_L^{-1}\widetilde{\alpha} \widetilde{\alpha}_R^{-1} \rho 
\widetilde{\alpha}_L^{-1}\widetilde{\beta}\widetilde{\alpha}_R^{-1}=\beta, 
$$
i.e. $\alpha\rho\beta$ if and only if $\widetilde{\alpha}\varrho \widetilde{\beta}$. 

\smallskip 

Now, suppose that $I=I_k^M$, for some $0\leqslant k\leqslant n-2$. 
Let $\alpha,\beta\in M$ be such that $|\im(\alpha)|\geqslant k+2$ and $\alpha\rho\beta$. 
Then, $\alpha\not\in I$ and so, by an above deduction, we have $\alpha\mathscr{H}\beta$. 
By contradiction, assume that $\alpha\neq\beta$. 
Let $\varepsilon_{\mathrm{min}}=\id_{\dom(\alpha)\setminus\{\min\dom(\alpha)\}}, 
\varepsilon_{\mathrm{max}}=\id_{\dom(\alpha)\setminus\{\max\dom(\alpha)\}}\in M$. 
Then, 
$\varepsilon_{\mathrm{min}}\alpha\rho\varepsilon_{\mathrm{min}}\beta$,   
$\varepsilon_{\mathrm{max}}\alpha\rho\varepsilon_{\mathrm{max}}\beta$ 
 and, since $|\im(\varepsilon_{\mathrm{min}}\alpha)|=|\im(\varepsilon_{\mathrm{max}}\alpha)|=k+1$, 
$\varepsilon_{\mathrm{min}}\alpha, \varepsilon_{\mathrm{max}}\alpha\not\in I$, 
whence 
$$
\mbox{$(\varepsilon_{\mathrm{min}}\alpha,\varepsilon_{\mathrm{min}}\beta)\in\mathscr{H}$ and  
$(\varepsilon_{\mathrm{max}}\alpha,\varepsilon_{\mathrm{max}}\beta)\in\mathscr{H}$,} 
$$
again by an above deduction.  
On the other hand, by Lemma \ref{pori}, we have $\im(\varepsilon_{\mathrm{min}}\alpha)\neq\im(\varepsilon_{\mathrm{min}}\beta)$ or 
$\im(\varepsilon_{\mathrm{max}}\alpha)\neq\im(\varepsilon_{\mathrm{max}}\beta)$, and so 
$$
\mbox{$(\varepsilon_{\mathrm{min}}\alpha,\varepsilon_{\mathrm{min}}\beta)\not\in\mathscr{H}$ or 
$(\varepsilon_{\mathrm{max}}\alpha,\varepsilon_{\mathrm{max}}\beta)\not\in\mathscr{H}$,} 
$$
which is a contradiction. 
Thus, $\alpha=\beta$. 

At this point, if  $0\leqslant k\leqslant n-3$ or if $k=n-2$, with either $M=\AOP_n$ and $n$ is odd or $M=\AOR_n$ and $n\not\equiv 2\mod 4$, 
we can conclude that 
$$
\mbox{$\rho=\conpi{J_{k+1}^M}^\varrho$, with $\varrho=\rho\cap(H_{\id_{\{1,\ldots,k+1\}}}^M\times H_{\id_{\{1,\ldots,k+1\}}}^M)$.} 
$$
On the other hand, 
if $k=n-2$, with either $M=\AOP_n$ and $n$ even or $M=\AOR_n$ and $n\equiv 2\mod 4$, we can conclude that 
$$
\mbox{$\rho=\contheta{J_{n-1}^{\mathscr{o},M}}^{\varrho_1} \cup \contheta{J_{n-1}^{\mathscr{e},M}}^{\varrho_2}$, 
with $\varrho_1=\rho\cap(H_{\id_{\{2,\ldots,n\}}}^M\times H_{\id_{\{2,\ldots,n\}}}^M)$
and $\varrho_2=\rho\cap(H_{\id_{\{1,\ldots,n-1\}}}^M\times H_{\id_{\{1,\ldots,n-1\}}}^M)$.}
$$

\smallskip 

Next, suppose that $I=I_{n-1}^M$. Then, clearly, we obtain 
$\rho=\conpi{J_n^M}^\varrho$, with $\varrho=\rho\cap(J_n^M\times J_n^M)$.

\smallskip 

Finally, let us suppose that $I=I^{\mathscr{o},M}_{n-1}$ [respectively, $I=I^{\mathscr{e},M}_{n-1}$]. 
Notice that, in this case, we have either $M=\AOP_n$ and $n$ even or $M=\AOR_n$ and $n\equiv 2\mod 4$. 
Let $\alpha,\beta\in J_n^M$ be such that $\alpha\rho\beta$. 
Let $\sigma=\alpha\beta^{-1}$ and, by contradiction, assume that $\sigma\neq\id_n$. 
Then, by Lemma \ref{d_2n} and by (\ref{g_n}) (this latter only necessary for $\AOP_4$), we have 
$\im(\id_{\Omega_n\setminus\{2\}}\sigma)\neq \Omega_n\setminus\{2\}$ or 
$\im(\id_{\Omega_n\setminus\{n\}}\sigma)\neq \Omega_n\setminus\{n\}$ 
[respectively, 
$\im(\id_{\Omega_n\setminus\{1\}}\sigma)\neq \Omega_n\setminus\{1\}$ or 
$\im(\id_{\Omega_n\setminus\{3\}}\sigma)\neq \Omega_n\setminus\{3\}$]. 
Hence, in particular, 
$$
\mbox{$(\id_{\Omega_n\setminus\{2\}}\sigma,\id_{\Omega_n\setminus\{2\}})\not\in\mathscr{H}$ 
or 
$(\id_{\Omega_n\setminus\{n\}}\sigma,\id_{\Omega_n\setminus\{n\}})\not\in\mathscr{H}$}
$$
[respectively, 
$$
\mbox{$(\id_{\Omega_n\setminus\{1\}}\sigma,\id_{\Omega_n\setminus\{1\}})\not\in\mathscr{H}$ 
or 
$(\id_{\Omega_n\setminus\{3\}}\sigma,\id_{\Omega_n\setminus\{3\}})\not\in\mathscr{H}$].}
$$
On the other hand, as $\alpha\rho\beta$, we get $\sigma=\alpha\beta^{-1}\rho\beta\beta^{-1}=\id_n$, whence 
$\id_{\Omega_n\setminus\{2\}}\sigma\rho\,\id_{\Omega_n\setminus\{2\}}$ 
and 
$\id_{\Omega_n\setminus\{n\}}\sigma\rho\,\id_{\Omega_n\setminus\{n\}}$ 
[respectively, $\id_{\Omega_n\setminus\{1\}}\sigma\rho\,\id_{\Omega_n\setminus\{1\}}$  
and 
$\id_{\Omega_n\setminus\{3\}}\sigma\rho\,\id_{\Omega_n\setminus\{3\}}$]. 
By an above deduction, since $\id_{\Omega_n\setminus\{2\}},\id_{\Omega_n\setminus\{n\}}\not\in I$ 
[respectively, 
$\id_{\Omega_n\setminus\{1\}},\id_{\Omega_n\setminus\{3\}}\not\in I$], we have 
$$
\mbox{$(\id_{\Omega_n\setminus\{2\}}\sigma,\id_{\Omega_n\setminus\{2\}})\in\mathscr{H}$ 
and 
$(\id_{\Omega_n\setminus\{n\}}\sigma,\id_{\Omega_n\setminus\{n\}})\in\mathscr{H}$} 
$$
[respectively, 
$$
\mbox{$(\id_{\Omega_n\setminus\{1\}}\sigma,\id_{\Omega_n\setminus\{1\}})\in\mathscr{H}$ 
and 
$(\id_{\Omega_n\setminus\{3\}}\sigma,\id_{\Omega_n\setminus\{3\}})\in\mathscr{H}$],}
$$ 
which is a contradiction. 
Thus, $\sigma=\id_n$ and so $\alpha=\beta$. 
Therefore, we obtain 
$$
\mbox{$\rho=\conpi{J_{n-1}^{\mathscr{e},M}}^\varrho$, with $\varrho=\rho\cap(H_{\id_{\{1,\ldots,n-1\}}}^M\times H_{\id_{\{1,\ldots,n-1\}}}^M)$} 
$$ 
[respectively, 
$$
\mbox{$\rho=\conpi{J_{n-1}^{\mathscr{o},M}}^\varrho$, with $\varrho=\rho\cap(H_{\id_{\{2,\ldots,n\}}}^M\times H_{\id_{\{2,\ldots,n\}}}^M)$],} 
$$
as required. 
\end{proof}

Notice that,  $\conpi{J_1^M}^\varrho$, with $\varrho$ the identity (and unique) congruence of $H^M_{\id_{\{1\}}}$, 
is the identity congruence of $M$. 
Moreover, with either $M=\AOP_n$ and $n$ even or $M=\AOR_n$ and $n\equiv 2\mod 4$, 
given congruences $\varrho_1$ and $\varrho_2$  of the groups $H^M_{\id_{\{2,\ldots,n\}}}$ and $H^M_{\id_{\{1,\ldots,n-1\}}}$, respectively,
then 
$\contheta{J_{n-1}^{\mathscr{o},M}}^{\varrho_1}\cup \contheta{J_{n-1}^{\mathscr{e},M}}^{\varrho_2}= 
\contheta{J_{n-1}^{\mathscr{e},M}}^{\varrho_2}$, if  $\varrho_1$ is the identity congruence of $H^M_{\id_{\{2,\ldots,n\}}}$, 
and 
$\contheta{J_{n-1}^{\mathscr{o},M}}^{\varrho_1}\cup \contheta{J_{n-1}^{\mathscr{e},M}}^{\varrho_2}= 
\contheta{J_{n-1}^{\mathscr{o},M}}^{\varrho_1}$, if  $\varrho_2$ is the identity congruence of $H^M_{\id_{\{1,\ldots,n-1\}}}$. 

\medskip 

Let $G$ be a cyclic group of order $m$. It is well known that
there exists a one-to-one correspondence between the subgroups of
$G$ and the (positive) divisors of $m$. Therefore, as all subgroups of $G$ are normal, there exists a one-to-one correspondence
between the congruences of $G$ and the (positive) divisors of $n$.
On the other hand, let $G$ be a dihedral group of order $2m$ ($m\geqslant 3$), 
let us say that $G$ is defined by the group presentation 
$
\langle \sigma,\tau \mid \sigma^m=1,\tau^2=1,\tau\sigma=\sigma^{-1}\tau\rangle
$. 
Then, it is also well known that its proper normal subgroups are:
\begin{description}
  \item \-- $\langle \sigma^2,\tau\rangle$, $\langle \sigma^2,\sigma\tau\rangle$ and $\langle \sigma^\frac{m}{p}\rangle$, 
  with $p$ a divisor of $m$, if $m$ is even;
  \item \-- $\langle \sigma^\frac{m}{p}\rangle$, with $p$ a divisor of $m$, if $m$ is odd.
\end{description}
See \cite{Dummit&Foote:1999} for more details.

Thus, since all maximal subgroups of $M$ are cyclic or dihedral, Theorem \ref{conM} gives us a complete characterization of the congruences of the monoid $M$. 

\section{Generators and rank of $\AOP_n$} \label{rAOP} 

We begin this section by recalling the following result from \cite{Fernandes:2000}. 

\begin{lemma}[{\cite[Proposition 3.1]{Fernandes:2000}}]\label{gib} 
Let $\alpha\in\POPI_n$. Then, there exist $0\< i\< n-1$ and $\beta\in\POI_n$ such that $\alpha=g^i\beta$. 
\end{lemma} 

Notice that, under the conditions of the previous lemma, $\beta$ must have the same rank as $\alpha$. 

\smallskip 

Now, let $X_i=\Omega_n\setminus\{i\}$, for $1\leqslant i\leqslant n$, and define $x_1,x_2,\ldots,x_n\in\POI_n$ by  
$$
x_1=\left\{\begin{array}{ll}
\transf{2&\cdots&n\\1&\cdots&n-1}=\transf{X_1\\X_n} & \mbox{if $n$ is odd}\\
\transf{2&\cdots&n-1&n\\1&\cdots&n-2&n}=\transf{X_1\\X_{n-1}} & \mbox{if $n$ is even}
\end{array}\right.,
\quad 
x_2=\left\{\begin{array}{ll}
\transf{1&3&\cdots&n-1&n\\1&2&\cdots&n-2&n}=\transf{X_2\\X_{n-1}} & \mbox{if $n$ is odd}\\
\transf{1&3&\cdots&n\\1&2&\cdots&n-1}=\transf{X_2\\X_n} & \mbox{if $n$ is even}
\end{array}\right.
$$
and
$$
x_i=\transf{1&\cdots&i-3&i-2&i-1&i+1&\cdots&n\\1&\cdots&i-3&i-1&i&i+1&\cdots&n}=\transf{X_i\\X_{i-2}}, ~\mbox{for $3\leqslant i\leqslant n$}. 
$$\label{genAO} 
Then, as proved in \cite{Fernandes:2024sub}, $\{x_1,x_2,\ldots,x_n\}$ is a set of generators of $\AO_n$ with minimum size (and so $\AO_n$ has rank $n$). 
Since $x_1,x_2,\ldots,x_n\in J_{n-1}$, we get $\AO_n\subseteq\langle J_{n-1}\rangle$. On the other hand, by Lemmas \ref{jn-1o} and \ref{jn-1e}, we obtain  
$J_{n-1}\subseteq\langle g,g_n^2\rangle$, if $n$ is odd, and $J_{n-1}\subseteq\langle g^2,g_1,g_n\rangle$, if $n$ is even. 
Furthermore, we have: 

\begin{theorem}\label{genAOP}
If $n$ is odd, then $\{g,g_n^2\}$ is a generating set of $\AOP_n$ with minimum size. 
If $n$ is even, then $\{g^2,g_1,g_n\}$ is a generating set of $\AOP_n$ with minimum size. 
In particular, $\AOP_n$ has rank $2$, if $n$ is odd, and $\AOP_n$ has rank $3$, if $n$ is even.  
\end{theorem} 
\begin{proof}
First, notice that, any generating set of $\AOP_n$ 
must contain at least an element of its group of units and an element of each $\mathscr{J}$-class of elements of rank $n-1$. 
Therefore, $\AOP_n$ has at least rank $2$, if $n$ is odd, and $\AOP_n$ has at least rank $3$, if $n$ is even.  

\smallskip 

Next, suppose that $n$ is an odd number. Then, $J_n=\langle g\rangle$ and, 
by Lemma \ref{gib} and the above observations, we have 
\begin{align*}
\{\alpha\in\AOP_n\mid |\im(\alpha)|\leqslant n-2\}=\{\alpha\in\POPI_n\mid |\im(\alpha)|\leqslant n-2\}
\subseteq\langle g,\{\beta\in\POI_n\mid |\im(\beta)|\leqslant n-2\}\rangle = \\ = 
\langle g,\{\beta\in\AO_n\mid |\im(\beta)|\leqslant n-2\}\rangle \subseteq 
\langle g,x_1,\ldots,x_n\rangle \subseteq 
\langle g,g_n^2\rangle, 
\end{align*}
and so $\AOP_n=\langle g,g_n^2\rangle$. 

\smallskip 

Now, suppose that $n$ is an even number. Then,  $J_n=\langle g^2\rangle$ and, as observed above, $J_{n-1}\subseteq\langle g^2,g_1,g_n\rangle$. 
Thus, let us take $\alpha\in\AOP_n$ such that $|\im(\alpha)|\leqslant n-2$ and $i\in\Omega_n\setminus\dom(\alpha)$. 

If $\alpha\in\AO_n$, as observed above, then $\alpha\in \langle J_{n-1}\rangle\subseteq\langle g^2,g_1,g_n\rangle$. 
So, suppose that $\alpha\in\AOP_n\setminus\AO_n$. 

Let $1\< j_1<j_2<\cdots<j_{n-1}\< n$ be such that $\Omega_n\setminus\{i\}=\{j_1,\ldots,j_{n-1}\}$. 
Let $1\< i_1<i_2<\cdots<i_k\< n-1$ ($2\< k=|\im(\alpha)|\< n-2$) and $1\< r \< k-1$ be such that $\dom(\alpha)=\{j_{i_1},j_{i_2},\ldots,j_{i_k}\}$ 
and $j_{i_r}\alpha > j_{i_{r+1}}\alpha$. 
Then, 
$$
j_{i_{r+1}}\alpha  < \cdots < j_{i_k}\alpha < j_{i_1}\alpha  < \cdots < j_{i_r}\alpha. 
$$

Let 
$$
\gamma=\transf{j_1&j_2&\cdots&j_{n-2}&j_{n-1}\\j_2&j_3&\cdots&j_{n-1}&j_1}.
$$
Then, clearly, $\gamma\in\POPI_n$ and, moreover, $\gamma\in J_{n-1}$ (notice that $\gd(\gamma)=\gi(\gamma)=i$), 
whence $\gamma\in \langle g^2,g_1,g_n\rangle$. Furthermore, 
$$
j_{i_{r+1}}\gamma^{n-1-j_{i_r}}  < \cdots < j_{i_k}\gamma^{n-1-j_{i_r}}  < j_{i_1}\gamma^{n-1-j_{i_r}}  < \cdots <  j_{i_r}\gamma^{n-1-j_{i_r}} 
$$
and so 
$$
\beta = \transf{j_{i_{r+1}}\gamma^{n-1-j_{i_r}} & \cdots & j_{i_k}\gamma^{n-1-j_{i_r}}  & j_{i_1}\gamma^{n-1-j_{i_r}}  & \cdots & j_{i_r}\gamma^{n-1-j_{i_r}} \\
j_{i_{r+1}}\alpha  & \cdots & j_{i_k}\alpha &  j_{i_1}\alpha  & \cdots &  j_{i_r}\alpha} \in \POI_n. 
$$
Since $|\im(\beta)|=|\im(\alpha)|\< n-2$, then $\beta\in\AO_n$ and thus $\beta\in \langle J_{n-1}\rangle\subseteq\langle g^2,g_1,g_n\rangle$. 
In addition, we have $\alpha= \gamma^{n-1-j_{i_r}}\beta$, whence $\alpha\in \langle g^2,g_1,g_n\rangle$. 

Therefore, we conclude that $\AOP_n=\langle g^2,g_1,g_n\rangle$, as required. 
\end{proof} 

\section{Generators and rank of $\AOR_n$}\label{rAOR}

Recall that $h\in\AOR_n$, if $n\equiv0\mod 4$ or $n\equiv1\mod 4$, and $hg\in\AOR_n$, if $n\equiv2\mod 4$. 
Therefore, let us define $h'=h$, if $n\equiv0\mod 4$ or $n\equiv1\mod 4$, and $h'=hg\in\AOR_n$, if $n\equiv2\mod 4$. 
Then, $h'\in\D_{2n}\setminus\C_n$ and so $h'^2=\id_n$.  

\smallskip 

Suppose that $n\not\equiv3\mod 4$ and let $\alpha\in\AOR_n\setminus\AOP_n$. 
Then, $h'\alpha\in\AOP_n$ and $\alpha=h'(h'\alpha)$, whence $\alpha\in\langle h',\AOP_n\rangle$. 
Thus, in view of Theorem \ref{genAOP} and the equality $g_1=hg_n^{n-2}h$, we have 
\begin{equation}\label{gen012}
\AOR_n=\left\{
\begin{array}{ll}
\langle g^2,h,g_n\rangle  & \mbox{if $n\equiv 0\mod 4$}\\
\langle g,h,g_n^2\rangle  & \mbox{if $n\equiv 1\mod 4$}\\
\langle g^2,hg,g_1,g_n\rangle  & \mbox{if $n\equiv 2\mod 4$}. 
\end{array}
\right.
\end{equation} 
Since any generating set of $\AOR_n$ must contain a generating set of its group of units and at least an element of each $\mathscr{J}$-class of elements of rank $n-1$, 
the generating sets presented in (\ref{gen012}) are of minimum cardinality. 

\smallskip 

Next, suppose that $n\equiv3\mod 4$. In this case, we prove below that $\AOR_n=\langle g,g_n^2,h_ng_n\rangle$.

Let $\alpha\in\AOR_n$. If $\alpha\in\AOP_n$, then $\alpha\in \langle g,g_n^2\rangle$, by Theorem \ref{genAOP}, 
whence $\alpha\in \langle g,g_n^2,h_ng_n\rangle$. So, suppose that $\alpha\in\AOR_n\setminus\AOP_n$. 
Take $k\in\Omega_n\setminus\dom(\alpha)$. Then, $n\not\in\dom(g^k\alpha)$ and so 
\begin{equation}\label{hngn2}
(h_ng_n)^2g^k\alpha = \id_{\{1,\ldots,n-1\}}g^k\alpha=g^k\alpha.
\end{equation}
On the other hand, as $\alpha\not\in\AOP_n$, we have $h_ng_ng^k\alpha\in\AOP_n=\langle g,g_n^2\rangle$. 
Hence, by (\ref{hngn2}), we get $g^k\alpha\in\langle g,g_n^2,h_ng_n\rangle$ and so $\alpha\in\langle g,g_n^2,h_ng_n\rangle$. 
Therefore, $\AOR_n=\langle g,g_n^2,h_ng_n\rangle$. 

Now, let $\mathcal{X}$ be a generating set of $\AOR_n$. Then, $\mathcal{X}$ must contain at least an element $Q_n$ and an element $\alpha\in Q_{n-1}\setminus J_{n-1}$. 
Since an element of $\mathcal{X}$ in $Q_n$ can be replaced by $g$, we can assume without loss of generality that $g\in\mathcal{X}$. 

Suppose that $\AOR_n=\langle g,\alpha\rangle$. 
Then, any element of $\AOR_n$ is of the form
$
g^{k_0}\alpha g^{k_1} \cdots \alpha g^{k_r}, 
$
with $r\geqslant0$ and $0\leqslant k_0,k_1,\ldots, k_r\leqslant n-1$. 

Observe that, if $\alpha g^s\alpha\in Q_{n-1}$, for some $0\leqslant s\leqslant n-1$, 
then $\im(g^s\alpha)=\dom(\alpha)$ and so $s$ is (the only element) such that $s \equiv (\gd(\alpha)-\gi(\alpha))\mod n$. 

Let $\beta\in \langle g_n^2\rangle$. Then, we must have $\beta=g^{\gd(\alpha)}\alpha (g^s\alpha)^{2t+1} g^{n-\gi(\alpha)}$, for some $t\geqslant0$. 
As $\dom(g^s\alpha)=\Omega_n\setminus\{\gi(\alpha)\}=\im(g^s\alpha)$, 
i.e. $g^2\alpha$ is a group element of $Q_{n-1}$ that reverses the orientation, then $(g^s\alpha)^2=\id_{\Omega_n\setminus\{\gi(\alpha)\}}$, 
whence $(g^s\alpha)^{2t+1}=g^s\alpha$ and so $\beta = g^{\gd(\alpha)}\alpha g^s\alpha g^{n-\gi(\alpha)}$. 
Therefore,  $\langle g_n^2\rangle=\{\beta\}$, which is a contradiction. 

Thus, $\mathcal{X}$ must include at least a third element and so we conclude that $\{g,g_n^2,h_ng_n\}$ is a generating set of $\AOR_n$ with minimum size.  

To sum up, we have shown that:

\begin{theorem}\label{genAOR}

If $n\equiv 0\mod 4$ {\rm[}respectively, $n\equiv 1\mod 4$, $n\equiv 2\mod 4$, $n\equiv 3\mod 4${\rm]}, 
then $\{ g^2,h,g_n\}$ {\rm[}respectively, $\{ g,h,g_n^2 \}$, $\{ g^2,hg,g_1,g_n \}$, $\{ g,g_n^2,h_ng_n \}${\rm]} 
is a generating set of $\AOR_n$ with minimum size. 
In particular, $\AOR_n$ has rank $4$, if $n\equiv 2\mod 4$, and $\AOR_n$ has rank $3$, if $n\not\equiv 2\mod 4$. 
\end{theorem} 

\subsection*{Acknowledgment}

This work is funded by national funds through the FCT - Funda\c c\~ao para a Ci\^encia e a Tecnologia, I.P., 
under the scope of the projects UIDB/00297/2020 (https://doi.org/10.54499/UIDB/00297/2020) 
and UIDP/00297/2020 (https://doi.org/10.54499/UIDP/00297/2020) (Center for Mathematics and Applications). 

\subsection*{Declarations} 

The author declares no conflicts of interest.

\end{document}